\documentclass{bull-l}

\usepackage{graphicx}
\usepackage{amssymb}
\usepackage{mathrsfs}
\usepackage{url}

\newtheorem{theorem}{Theorem}[section]
\newtheorem{lemma}[theorem]{Lemma}
\newtheorem{proposition}[theorem]{Proposition}
\newtheorem{corollary}[theorem]{Corollary}

\theoremstyle{definition}
\newtheorem{definition}[theorem]{Definition}
\newtheorem{example}[theorem]{Example}
\newtheorem{examples}[theorem]{Examples}

\theoremstyle{remark}
\newtheorem{remark}[theorem]{Remark}

\numberwithin{equation}{section}

\newcommand\SU{{\mathrm{SU}}}
\newcommand\GL{{\mathrm{GL}}}

\newcommand\U{{\mathrm U}}

\newcommand\Z{{\mathbb Z}}

\newcommand\N{{\mathbb N}}
\newcommand\C{{\mathbb C}}
\newcommand\R{\mathbb{R}}
\newcommand\Hom{\mathrm{Hom}}
\newcommand\End{\mathrm{End}}
\newcommand\Tr{\mathrm{Tr}}
\newcommand\tr{\mathrm{tr}}
\newcommand\hol{\mathrm{hol}}

\newcommand\E{\mathbb{E}}

\newcommand\vol{\mathrm{vol}}

\newcommand\Gbb{\mathbb{G}}
\newcommand\Cgg{{\mathscr{C}_{\mathbb{G}}^G}}
\newcommand\YM{{\mathrm{YM}}}
\newcommand\Pfr{{\mathscr{P}}}
\newcommand\Gfr{{\mathscr{G}}}
\newcommand\Tbb{{\mathbb{T}}}
\newcommand\Ufr{{\mathscr{U}}}

\begin{document}

\title[$\mathrm{YM}_2$ via integrable probability]{Two-dimensional Yang--Mills theory via integrable probability}


\author{Thibaut Lemoine}
\address{Institut de Recherche Math\'ematique Avanc\'ee, UMR 7501, Universit\'e de Strasbourg, 7 rue René Descartes,
67084 Strasbourg, France}
\email{thibaut.lemoine@unistra.fr}
\thanks{}



\date{}

\dedicatory{}

\begin{abstract}
In this paper, we review the construction and large $N$ study of the continuous two-dimensional Yang--Mills theory with gauge group $\mathrm{U}(N)$ through probability, combinatorics and representation theory. In the first part, we define the continuous Yang--Mills measure using Markovian holonomy fields, following a construction by L\'evy, then we show in the second part how to derive the character expansion of the partition function for any compact structure group from this setting. We continue with two developments obtained in the last few years by Dahlqvist, Lemoine, L\'evy and Ma\"ida with similar approaches with respect to the partition function: its large-$N$ asymptotics on all compact surfaces for the structure group $\mathrm{U}(N)$, and its $\frac{1}{N}$ expansion on a torus with an interpretation in terms of random surfaces.
\end{abstract}

\maketitle


\section{Introduction}

Yang--Mills theory, named after two physicists who introduced it in the '50s, is a non-abelian gauge theory that aims to describe, in a unified way, fundamental interactions between elementary particles. It is a quantum field theory based on the Yang--Mills action, a functional on the space of connections of a principal bundle. In the framework of the \emph{Standard model} of particle physics, it describes three of the four fundamental interactions:
\begin{enumerate}
\item Electromagnetism,
\item Weak nuclear force,
\item Strong nuclear force.
\end{enumerate}
There have been many mathematical developments about Yang--Mills theory since the '70s. Although the rigorous mathematical construction of a quantum Yang--Mills theory on a continuous four-dimensional spacetime remains an open problem (it is even one of the seven \emph{Millenium Prize Problems} from Clay Mathematics Institute \cite{JW}), there have been significant advances in the understanding of the discrete theory in any dimension, as well as the continuous theory in two dimensions. We refer the reader to \cite{GP,DEFJKMMW} for general background in quantum field theory and Yang--Mills theory. This survey, mainly based on a probabilistic construction of Thierry L\'evy \cite{Lev03,Lev10} and recent papers of the author together with Antoine Dahlqvist and Myl\`ene Ma\"ida \cite{Lem,DL,Lem3,LM,LM2}, is written to target a broad audience, and aims to be an accessible introduction to the subject. We focus on an approach based on integrable probability, an evergrowing field that lies at the interface of integrable systems, discrete probabilistic models and representation theory; we refer for instance to \cite{BP17,Zyg} for recent reviews about integrable probability in the context of the so-called KPZ universality (which is seemingly unrelated but can be tackled with similar techniques). In the course of these notes, we will play with many objects from enumerative combinatorics, which have independent interest: maps, partitions or even ramified coverings. To keep a reasonable length, we will only present them superficially, and we will detail proofs when they are short and/or important enough for the understanding. Let us briefly summarize the scope of this paper, by starting with {\bf what is omitted}.

\begin{itemize}
\item We shall not discuss the lattice gauge theory on $\Z^d$, which is well-defined in any dimension (especially in the physically relevant dimension 4). For recent surveys and results, we refer to \cite{Cha19,Cha20,CPS} and references therein.
\item We will not describe the interplay between the large-$N$ asymptotics of the two-dimensional Yang--Mills measure and free probability theory, which is incarnated by the \emph{master field} described by Singer \cite{Sin}. We refer the interested reader to the papers \cite{Lev17,DN,DL,PPSY,DL2} where the master field was constructed on various surfaces, as well as the introductory notes \cite{Lev20}.
\item We will not explore the links between Yang--Mills theory on a compact surface of genus $g\geq 2$ and the integration on moduli spaces of flat connections, which was initiated by Witten \cite{Wit91}. We refer to \cite{Sen02,Sen03} for detailed mathematical developments.
\item Finally, we will not discuss the recent results about stochastic quantization of the Yang--Mills theory in dimensions 2 and 3 \cite{CCHS22,CheShe23,CCHS24}, which are of independent interest but develop completely different aspects -- see \cite{Che22} for a review.
\end{itemize}

\subsection{Formal definition and difficulties}

Let us first recall the historical definition of the Yang--Mills measure. We will assume basic knowledge of differential geometry (Riemannian manifolds, Lie groups, principal bundles) in this subsection, but it will not be essential for the rest of the paper and any reader unfamiliar with these notions can simply skip this subsection without harm.

The two main ingredients of the theory are a spacetime and a group of symmetries that describes gauge transformations. More precisely, we take a Riemannian manifold $(M,g)$ and a compact Lie group $G$, called \emph{structure group}. We assume that the corresponding Lie algebra $\mathfrak{g}$ is endowed with a $G$-invariant inner product $\langle\cdot,\cdot\rangle$. The spacetime and structure group can be put together into a $G$-principal bundle $\pi:P\to M$. The interactions are represented by the \emph{Yang--Mills Lagrangian}, which is a function on the space $\mathcal{A}$ of connections on $P$. The latter is an affine space of direction
\[
\Omega^1(M)\otimes\mathrm{ad}(P),
\]
the space of $\mathrm{ad}(P)$-valued 1-forms on $M$. The Yang--Mills action is then defined by
\begin{equation}
S_\YM(\omega) = \int_M\langle\Omega_\omega\wedge\star\Omega_\omega\rangle,
\end{equation}
where $\star$ is the Hodge star operator, $\langle\cdot,\cdot\rangle$ is the structure on $\mathrm{ad}(P)$ induced by the invariant inner product on $\mathfrak{g}$, and $\Omega_\omega$ is the curvature of the connection $\omega\in\mathcal{A}$.

The \emph{Yang--Mills measure} on $M$ with structure group $G$ is the measure \emph{formally defined} by
\begin{equation}
d\mu_\YM(\omega) = \frac{1}{Z_\YM}e^{\frac{\mathrm{i}}{2T}S_\YM(\omega)}\mathcal{D}\omega,
\end{equation}
where  $\mathcal{D}\omega$ is the Lebesgue measure on $\mathcal{A}$, $T>0$ is a \emph{coupling constant}, and $Z_\YM$ is a normalization constant that puts the total mass to 1. A first obvious problem is that such measure, assuming that it exists, is not positive. However, by analytic continuation, it can be turned into a true probability measure by operating a \emph{Wick rotation}:
\begin{equation}\label{eq:YM_heuristique_eucl}
d\mu'_\YM(\omega) = \frac{1}{Z_\YM}e^{-\frac{1}{2T}S_\YM(\omega)}\mathcal{D}\omega.
\end{equation}
The resulting object is called the \emph{Euclidean Yang--Mills measure}. Many critical issues remain: for instance, there exists no Lebesgue measure on $\mathcal{A}$, and even if it was the case, there would be no reason that the measure has finite mass. Even if we reduce the space of integration by taking the quotient by gauge symmetries, there is no well-defined corresponding measure on the quotient. In fact, such measure can be defined on an appropriate compactification \cite{AL}, but nothing guarantees that the Yang--Mills measure has a density with respect to it. A partial solution was however obtained by Ashtekar \emph{et al.} \cite{ALMMT} (in the sense that they were able to compute all relevant observables within this framework, on the plane and the cylinder, with structure group $\SU(2)$).

The first complete, rigorous constructions of path integrals for the continuous two-dimen\-sional Yang--Mills theory (also denoted by $\mathrm{YM}_2$) are due to Gross, Driver, King and Sengupta \cite{GKS,Dri,Sen97} using mainly stochastic analysis. Another construction was later obtained by L\'evy \cite{Lev03,Lev10}, using a combination of discrete gauge theory developed by Witten \cite{Wit91} and continuous Markov processes. We develop this approach in Section~\ref{sec:YM}: if we replace $M$ by a graph $\Gbb$ and the principal bundle $\pi:P\to M$ by a discrete trivial bundle $\pi':\Gbb\times G\to \Gbb$, we can make sense of appropriate integrals with respect to~\eqref{eq:YM_heuristique_eucl}. The collection of measures on appropriate graphs will be the finite-dimensional marginals of a continuous Markov process called the \emph{Yang--Mills holonomy field}, which can then be constructed by the Kolmogorov extension theorem. The distribution of this process is, in some sense, the ``right" continuous Yang--Mills measure (see the end of \S\ref{sec:continuous_extension} about why this is true).

\subsection{Harmonic analysis and partition function}

In two dimensions, the first interesting quantity to study is the partition function $Z_\YM$. The partition function for a structure group $G$ on a compact connected oriented surface of genus $g$ and area $t$ only depends on $G,g$ and $t$, and can be computed explicitly as an integral on the group $G$ using the previously introduced setting; we will denote by $Z_G(g,t)$ this partition function. Section~\ref{sec:AH_groupes} will be devoted to deriving a Fourier expansion of $Z_G(g,t)$, using general results of the representation theory of compact groups. We will thus recover, in a self-contained way, a formula initially found by Rusakov \cite{Rus2}.

\subsection{Large-$N$ limit and random surfaces}

Since the seminal paper of 't Hooft \cite{Hoo74}, it is known that Yang--Mills theory for $\SU(N)$ as $N$ goes to infinity gives an understanding of gauge theory for $\SU(3)$ at strong coupling, which is related to the question of the confinement of quarks (it is also named ``multicolor QCD" by physicists, because $\SU(3)$ Yang--Mills corresponds to the quantum chromodynamics, and $N=3$ corresponds to the number of colors of quarks and gluons). The question of the asymptotic behavior of observables of the theory for matrix Lie groups with size going to infinity has been developed since then under the name ``large-$N$ regime". We will consider the case of $G=\U(N)$ in Section~\ref{sec:SchurWeyl}, and describe the asymptotics of $Z_{\U(N)}(g,t)$ for all genera when $N$ tends to infinity, giving a proof of several results conjectured by physicists in the '90s \cite{Gro,DK,Rus,Dou}.

Another aspect of the large-$N$ regime, already present in \cite{Hoo74} but mainly developed in the two-dimensional setting by Gross and Taylor \cite{Gro,GT,GT2}, is the relationship between Yang--Mills theory and string theory, namely a theory of random surfaces/graphs. They conjectured that the partition function and Wilson loop expectations on a compact surface $\Sigma_g$ of genus $g\geq 0$ with structure group $\SU(N)$ could be rewritten as sums (or integrals) over ramified coverings $X\to\Sigma_g$. There have been many combinatorial developments related to this interaction, nowadays known as \emph{gauge/string duality}\footnote{In fact, such dualities are also conjectured for other types of gauge theories, such as supersymmetric Yang--Mills theory. The two most cited papers from high energy physics ever -- up to date -- are related to another kind of gauge/string duality called AdS/CFT correspondence \cite{Mal,Wit98}.}, and it is still an active field in physics \cite{PS,AKS}, but everything is written in terms of formal power series. This gauge/string duality is also conjectured to hold for other classical groups, such as $\U(N),\mathrm{SO}(N)$ or $\mathrm{Sp}(N)$, and the first proof of a true asymptotic expansion related to this gauge/string duality on compact surfaces was achieved for $\U(N)$ when the base space is a torus \cite{LM}. A complete solution, including the asymptotic expansion for all compact classical groups and a random surface formula, is provided in the recent paper \cite{LM2}. We will go through the main results and ideas in Section~\ref{sec:gauge-string}.

As a side note, another kind of gauge/string duality has been explored for the \emph{lattice Yang--Mills theory} in any dimension \cite{Jaf16,Cha19b,CPS,BCSK}. The latest works in this direction are partly based on the Weingarten calculus, which is another combinatorial toolbox coming from representation theory of compact classical groups. In these works, the authors obtained elegant expressions of Wilson loop expectations in terms of convergent surface sums, both in the finite $N$ and the large-$N$ regime.



\section{Two-dimensional Yang--Mills measure}\label{sec:YM}

\subsection{Surfaces and maps}

In these notes, a \emph{surface} will always denote a smooth real differential manifold of dimension 2. We shall start with a definition and some properties of topological maps, which will play the role of our discretized spacetime. We refer to \cite{LZ,Eyn} for more details on maps.

\begin{definition}
A \emph{graph} is a triple $\Gbb=(V,E,I)$ where $V$ and $E$ are respectively the sets of \emph{vertices} and \emph{edges}, and an \emph{incidence relation} $I\subset V\times E$, such that the cardinal of $\{v\in V: (v,e)\in I\}$ is 1 or 2 for all $e\in E$.
\end{definition}

The incidence relationship is an abstract encoding of the intuitive idea that the vertices are at the boundaries of edges. If $\Gbb=(V,E,I)$ is a graph,
\begin{enumerate}
\item Two edges $e_1,e_2\in E$ are \emph{adjacent} if they are incident to the same vertex $v\in V$. They form a \emph{double edge} if they are incident to the two same vertices $v_1,v_2\in V$;
\item An edge $e\in E$ is a \emph{loop} if it is incident to only one vertex.
\end{enumerate}

\begin{definition}\label{def:carte}
A \emph{topological map} is a graph (possibly with loops and double edges) $\Gbb=(V,E,I)$ endowed with an embedding $\theta:\Gbb\to\Sigma$, where $\Sigma$ is a surface, such that:
\begin{itemize}
\item The images of two distinct vertices $v_1, v_2\in V$ by $\theta$ are distinct points of $\Sigma$,
\item The images of edges $e\in E$ are continuous curves $\theta_e:[0,1]\to\Sigma$ that only meet at their endpoints $\underline{e}=\theta_e(0)$ and $\overline{e}=\theta_e(1)$,
\item The complement in $\Sigma$ of the \emph{skeleton} $\mathrm{Sk}(\Gbb)=\bigcup_{e\in E}\theta_e$ of $\Gbb$ divides into one or several connected components $f_1,\ldots,f_k$, each simply connected, called \emph{faces}.
\end{itemize}
We denote by $F=\{f_1,\ldots,f_k\}$ the set of faces of the map.
\end{definition}

Usually, we will denote $\Gbb=(V,E,F)$ the map, and the embedding will be implicit. An example of map is given in Figure~\ref{fig:map_torus}. Edges $e\in E$ have an orientation by default, in the sense that the curve $\theta(e)$ inherits the orientation of its parametrized path, and we shall consider the action of $\{\pm1\}$ on $E$ given by:
\begin{itemize}
\item $e^1$ is simply $e$,
\item $e^{-1}$ is the edge $e$ with reverse orientation, i.e. corresponds to the paramet\-rized path $\theta_e^{-1}:t\mapsto \theta_e(1-t)$.
\end{itemize}

\begin{figure}[t!]
    \centering
    \includegraphics[width=\linewidth]{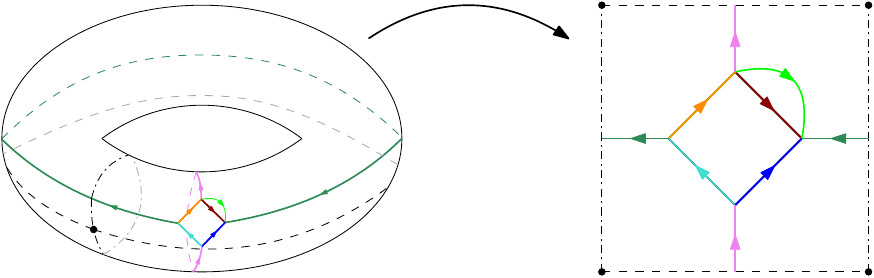}
    \caption{\small An oriented map of genus 1, embedded in a torus. It has 4 vertices, 7 edges and 3 faces.}
    \label{fig:map_torus}
\end{figure}

A topological map is an example of two-dimensional CW-complex in algebraic topology: its vertices (resp. edges, faces) are its 0-cells (resp. 1-cells, 2-cells), and the skeleton of the map is exactly the skeleton of the complex. Moreover, if the underlying surface is oriented, then the faces naturally inherits this orientation.

\begin{examples}
Among the most studied maps, let us mention triangulations (resp. quadrangulations): they are maps where all faces are triangles (resp. quadrilaterals). The study of random triangulations/quadrangulations is still a very active field, see for instance \cite{GM21,BL}.
\end{examples}

Topological maps can be put into equivalence classes up to orientation-preserving homeomorphisms, so that they do not depend on the embedding: this leads to the notion of \emph{combinatorial map}. These objects have attracted a growing interest in the last 30 years\footnote{Physicists were actually interested in them much earlier, starting with the work of 't Hooft \cite{Hoo74}, but they became popular in the mathematics community after Kontsevich used them to prove Witten's conjecture \cite{Kon}.}. A beautiful application of the study of maps is the classification of surfaces, see \cite{Lab}.

\begin{theorem}[Classification of surfaces]
Let $\Sigma$ be a compact connected surface without boundary.
\begin{itemize}
\item If $\Sigma$ is orientable, then it is homeomorphic to one of the following:
\begin{enumerate}
\item A sphere,
\item The connected sum of $g\geq 1$ tori.
\end{enumerate}
In this case, it is said to have \emph{genus} $g$, with $g=0$ by convention if it is homeomorphic to a sphere.
\item If $\Sigma$ is not orientable, then it is homeomorphic to the connected sum of $k\geq 1$ real projective planes.
\end{itemize}
\end{theorem}

This theorem is proved by showing that any such surface can be obtained as the gluing of a polygon along its oriented sides (or two disks in the case of the sphere). The polygon is a $4g$-gon in the case of an orientable surface of genus $g\geq 1$, which can be seen as a map with one face, called \emph{fundamental map} (as the polygon is often called \emph{fundamental polygon}). The sides of the polygon, in counterclockwise order, can be chosen to be
\[
[a_1,b_1]\ldots[a_g,b_g]=a_1b_1a_1^{-1}b_1^{-1}\ldots a_gb_ga_g^{-1}b_g^{-1},
\]
and the $a_i,b_i$ represent oriented edges of the map, with the convention that two edges of the same label (up to orientation) are glued together. An example of map in genus 2 is given in Figure~\ref{fig:map_genus_2} below. Going from the planar representation to the the embedded representation is an exercise of visualization: say that one wants to follow a small loop oriented counterclockwise (in black) around the vertex of the map, then record all oriented edges that are crossed from the right. The loop crosses the following edges: $b_1,a_1^{-1},b_1^{-1},a_1,b_2,a_2^{-1},b_2^{-1},a_2$.
\begin{figure}[b!]
    \centering
    \includegraphics[width=\linewidth]{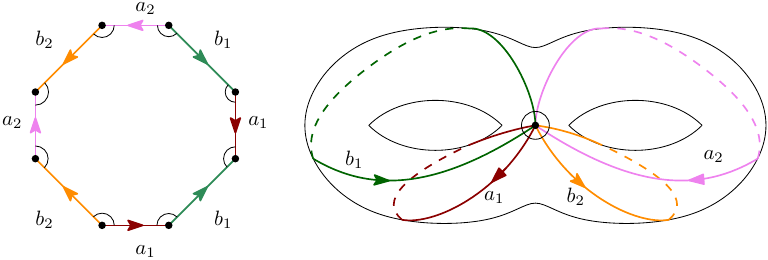}
    \caption{\small A map with one face in genus 2, traced in the plane (left) with edges of the same color pairwise identified, and traced in a genus 2 surface (right). It only has one vertex, marked by a black dot.}
    \label{fig:map_genus_2}
\end{figure}
The homotopy classes of the edges of the fundamental map generate the fundamental group of the underlying surface
\[
\pi_1(\Sigma)=\langle a_1,b_1,\ldots,a_g,b_g\ \vert \ [a_1,b_1]\ldots[a_g,b_g]=1\rangle,
\]
where we also denoted by $a_i,b_i$ the homotopy classes of the corresponding edges.

For any map $\Gbb=(V,E,F)$ embedded in a closed orientable surface $\Sigma$ of genus $g\geq 0$, the numbers of vertices, edges and faces are related through Euler's formula
\begin{equation}\label{eq:Euler}
\vert V\vert-\vert E\vert + \vert F\vert = 2-2g = \chi(\Sigma),
\end{equation}
where $\chi(\Sigma)$ is the Euler characteristic of the underlying surface.

In this paper, we will only consider closed orientable surfaces; a more general construction is detailed in \cite{Lev10}. We will henceforth denote by a \emph{surface of genus} $g$, a compact, connected, oriented surface of genus $g$.

\begin{definition}
Let $\Gbb=(V,E,F)$ be a topological map. A \emph{path} $\gamma$ in $\Gbb$ is either a single vertex $v\in V$ (it is a \emph{constant path}), or a concatenation of edges $e_1\cdots e_n$, where $\overline{e_i}=\underline{e_{i+1}}$ for all $1\leq i \leq n-1$. In the second case, we write $\underline{\gamma}=\underline{e_1}$ and $\overline{\gamma}=\overline{e_n}$ its endpoints, and $\ell(\gamma)=n$ its length. The path is a \emph{loop} if $\underline{\gamma}=\overline{\gamma}$.
\end{definition}

Constant paths in a map are also considered as loops by convention. For any loop $\gamma$ in $\Gbb,$ we call \emph{basepoint} the vertex $\underline{\gamma}=\overline{\gamma}.$ We denote by $\mathrm{P}(\Gbb)$ (resp. $\mathrm{P}_v(\Gbb)$, $\mathrm{L}(\Gbb)$, $\mathrm{L}_v(\Gbb)$) the set of paths (resp. paths of base $v\in V$, loops, loops of base $v\in V$) in $\Gbb$.

\subsection{Heat kernel on compact Lie groups}

\begin{definition}
A \emph{Lie group} is a group $G$ endowed with the structure of a smooth real manifold, such that the multiplication and the inversion are smooth.
\end{definition}

In what follows, we will denote by $1=1_G$ its neutral element. We need two objects from differential and integral calculus on $G$: the Laplace--Beltrami operator and the Haar measure. To be more specific, the Haar measure only requires a structure of a topological group, and the Laplacian a structure of a Riemannian manifold.

\begin{definition}
Let $G$ be a locally compact topological group. A nonnegative Radon measure $\mu$ on $G$ is \emph{left-invariant} if it satisfies for all $f\in\mathscr{C}_c(G)$
\[
\int_G f(gx)d\mu(x) = \int_G f(x)d\mu(x),\quad \forall g\in G.
\]
It is \emph{right-invariant} if it satisfies for all $f\in\mathscr{C}_c(G)$
\[
\int_G f(xg)d\mu(x) = \int_G f(x)d\mu(x),\quad \forall g\in G.
\]
A positive left-invariant (resp. right-invariant) on $G$ is called \emph{left Haar measure} (resp. \emph{right Haar measure}).
\end{definition}

A theorem due to von Neumann and Kakutani \cite{vN,Kak} states that for any compact topological group $G$, there exists a unique (up to a constant factor) Haar measure, which is simultaneously a left and a right Haar measure, and which has finite mass. We denote by $dg$ the corresponding probability measure.

If $G$ is a compact topological group and $f,g:G\to\R$ are Borel functions, their convolution is defined by
\[
f*g(x)=\int_G f(y)g(y^{-1}x)dy.
\]

\begin{definition}
Let $G\subset\GL(N,\C)$ be a compact matrix Lie group, whose Lie algebra $\mathfrak{g}=T_1G$ is endowed with a $G$-invariant inner product $\langle\cdot,\cdot\rangle$. Let $\{X_1,\ldots,X_d\}$ be an orthonormal basis of $\mathfrak{g}$ for this inner product. The \emph{Laplace--Beltrami} operator on $G$ is the differential operator $\Delta_G:\mathscr{C}^2(G)\to\C$ defined by
\begin{equation}
\Delta_G f(g) = \sum_{i=1}^d \frac{d^2}{dt^2}\bigg\vert_{t=0} f\left(ge^{tX_i}\right).
\end{equation}
The \emph{heat kernel} on $G$ is the convolution semigroup $(p_t)_{t\geq 0}$ solution of
\begin{equation}
\frac{d}{dt}p_t(g) =\frac12\Delta_G f(g),\quad \forall g\in G,
\end{equation}
with initial condition $\lim_{t\to 0^+} p_t = \delta_1$, where the limit has to be understood in the sense of distributions.
\end{definition}

From a probabilistic point of view, the heat kernel is the distribution of a Brownian motion on $G$ starting from 1. This approach has been fruitful to study the heat kernel on compact groups \cite{Lev08,LevMai}. In Section~\ref{sec:AH_groupes} we will use representation theory to obtain a spectral decomposition of the heat kernel on $G=\U(N)$ associated with this inner product.

\subsection{Discrete Yang--Mills measure}

As we have seen in the introduction, the Yang--Mills measure is formally a measure on the space of connections on the underlying spacetime. In the case of a graph, this space becomes finite-dimensional, so that the rigorous definition is within reach.

\begin{definition}\label{def:discrete_jauge}
Let $\Gbb=(V,E,I)$ be a graph, and $G$ be a compact matrix Lie group.
\begin{enumerate}
\item The \emph{trivial} $G$-\emph{principal bundle} above $\Gbb$ is $P=G^V$.
\item A \emph{section} of $P$ is simply an element of $P$, i.e. a function $\sigma:V\to G$. 
\item A \emph{connection} on $P$ is a function $\omega:E\to G$
\[
\omega:\left\lbrace\begin{array}{ccc}
E & \longrightarrow & G\\
e & \longmapsto & \omega(e),
\end{array}\right.
\]
that satisfies
\[
\omega(e^{-1})=\omega(e)^{-1},\quad \forall e\in E.
\]
\item Given a path $\gamma=e_1\cdots e_n$ in $\Gbb$, the \emph{holonomy} of the connection $\omega$ along $\gamma$ is given by
\[
\hol(\omega,\gamma):=h_\gamma = \omega(e_1)\ldots\omega(e_n).
\]
\item If moreover $\Gbb$ is a map, the \emph{curvature} of $\omega$ is the function
\[
\Omega^\omega:\left\lbrace\begin{array}{ccc}
F & \longrightarrow & G\\
f & \longrightarrow & \hol(\omega, \partial f),
\end{array}\right.
\]
where $\partial f$ is the boundary of $f$ oriented counterclockwise. We say that $\omega$ is \emph{flat} if $\Omega^\omega(f)=1$ for all $f\in F$.
\end{enumerate}
\end{definition}

There is a caveat in the previous definition: normally, connections and their curvature are respectively defined as $\mathfrak{g}$-valued $1$-forms and $2$-forms, whereas they are defined in Definition~\ref{def:discrete_jauge} as $G$-valued forms. This conversion from the Lie algebra to the Lie group turns the linear structure of connections into a multiplicative structure, and that is because we implicitly identified connections with their holonomies along edges. The curvature of a connection depends on a choice of basepoint for the boundary of each face, but all choices of basepoint give values of curvatures that are in the same conjugacy class in $G$. In the end, we will consider configurations to be discrete connections modulo gauge transformations, and in this quotient the definition of curvature is therefore uniquely defined. Furthermore, observe that one can interpret the curvature as a $G$-valued 2-form on $\Gbb$. More generally, 0-forms (resp. 1-forms, 2-forms) correspond to functions on the 0-cells (resp. 1-cells, 2-cells) of the map $\Gbb$ seen as a CW complex. For more details on this viewpoint, see \cite[Section 2]{DL2} for maps, or \cite[Section 2]{Cha20} for the lattice $\Z^d$.

\begin{examples}$ $

\begin{itemize}
\item If $\Gbb=\Z^d$ and $G=\Z_2$, a $G$-principal bundle endowed with a connection defines the model of \emph{Ising gauge theory} \cite{Cha20}: the sections of the bundle define a spin configuration, and the connection defines an interaction between adjacent spins, with the condition that the connection is flat. A generalized model has been recently studied in \cite{FLV}.
\item If we take $G=\U(1)$ and an arbitrary $\Gbb$, we recover the configuration space of \emph{cycle-rooted spanning forests} (CRSF) studied by Kenyon \cite{Ken11}.
\end{itemize}
\end{examples}

The space of connections is clearly in bijection with $G^{\vert E\vert}$, which is a compact Lie group. The \emph{uniform measure} on connections is taken to be the Haar probability measure on $G^{\vert E\vert}$.

\begin{definition}
Let $\Gbb=(V,E,I)$ be a graph endowed with a $G$-connection $h$. A \emph{gauge transformation} is a function $j:V\to G$. The gauge group $G^V$ acts on $G^E$ as follows:
\begin{equation}
(j\cdot \omega)(e) = j(\overline{e})^{-1}\omega(e) j(\underline{e}),
\end{equation}
and the \emph{configuration space} of the discrete gauge theory on $\Gbb$ with structure group $G$ is the quotient $\Cgg=G^E/G^V$ by this action.
\end{definition}

Intuitively, the discrete Yang--Mills measure would be a measure on $G^E$ invariant by this action, and is expected to descend to a measure on $\Cgg$. We can endow the set $\Omega^1(\Gbb,G)=G^E$ with two different sigma-algebras:
\begin{itemize}
\item The \emph{cylindrical} sigma-algebra $\mathcal{C}$ is the smallest sigma-algebra that makes the applications $h\mapsto h_\ell$ measurable for all loop $\ell$ in $\Gbb$;
\item The \emph{invariant} sigma-algebra $\mathcal{I}$ is the smallest sigma-algebra that makes the applications $f(h_{\ell_1},\ldots,h_{\ell_n})$ measurable, for any $n\geq 1$, any $f:G^n\to\R$ invariant by diagonal action of $G$ by conjugation, and any collection $(\ell_1,\ldots,\ell_n)$ of loops on $\Gbb$ with same basepoint.
\end{itemize}

Naturally, only the second one can be expected to descend to quotient into a sigma-algebra on $G^E/G^V$. By invariance of the Haar measure by translation, the Haar measure $d\omega$ on $\Omega^1(\Gbb,G)$ also descends to a measure on $G^E/G^V$, that we shall also denote by $d\omega$.

\begin{remark}
This configuration space is related to the notion of \emph{spin networks} \cite{Bae}. If one does not assume $G$ compact, this setting remains well-defined, and seems to be suitable for a formalization of quantum gravity \cite{Bae2,FreiLiv}. However, in this case, there is no more uniform probability measure: there are either Haar measures with infinite mass or probability measures that are not uniform.
\end{remark}

We are now able to define the discrete Yang--Mills measure.

\begin{definition}
Let $\Gbb=(V,E,F)$ be a topological map endowed in a surface of genus $g\geq 0$, endowed with a smooth measure of area $\vol$ and with total area $t>0$, and $G$ be a compact matrix Lie group. The \emph{Yang--Mills measure} on $\Gbb$ with structure group $G$ is the measure $\mu_{\Gbb,G,\Sigma,\vol}$ on $\Omega^1(\Gbb,G)$ defined by
\begin{equation}\label{eq:DriverSengupta}
\mu_{\Gbb,G,\Sigma,\vol}(d\omega) = \frac{1}{Z_G(g,t)}\prod_{f\in F}p_{\vol(f)}(\Omega^\omega(f))d\omega,
\end{equation}
where
\begin{equation}\label{eq:PartitionFunction}
Z_{\Gbb,G}(g,t) = \int_{G^E} \prod_{f\in F} p_{\vol(f)}(\Omega^\omega(f))d\omega
\end{equation}
is the \emph{partition function}.
\end{definition}

Equation~\eqref{eq:DriverSengupta} is named \emph{Driver--Sengupta formula}, from the two mathematicians who defined it first: Driver in the plane \cite{Dri} and Sengupta on compact surfaces \cite{Sen97}. As the heat kernel on $G$ is invariant by conjugation, the measure does not depend on the parametrization of the boundary of each face for the definition of the curvature, and the measure descends to a measure on the configuration space $\Cgg$. By construction, the measure is also invariant by area-preserving diffeomorphisms.

\begin{example}
Consider the map given in Figure~\ref{fig:map_example}, which has genus 1.
\begin{figure}[b!]
    \centering
    \includegraphics[width=0.4\linewidth]{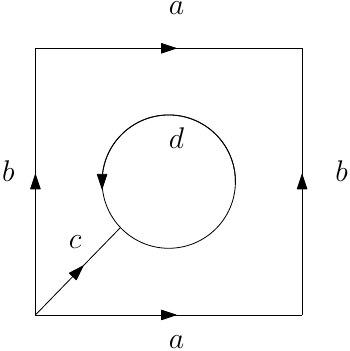}
    \caption{\small A map of genus 1. The edges with same label are glued together.}
    \label{fig:map_example}
\end{figure}
It is described by $\Gbb=(V,E,F)$, where $V=\{v_1,v_2\}$, $E=(a,b,c,d)$ and $F=\{f_1,f_2\}$ where $f_1$ has boundary $\partial f_1=d$ and $f_2$ has boundary $\partial f_2=d^{-1}c^{-1}aba^{-1}b^{-1}c$. If we set $t_i=\vol(f_i)$ for $i\in\{1,2\}$, as well as $t=t_1+t_2$ the total area of the underlying surface, the Driver--Sengupta formula becomes
\[
\mu_{\Gbb,G,\Sigma,\vol}(d\omega)=\frac{1}{Z_G(1,t)}p_{t_1}(\omega(d))p_{t_2}(\hol(\omega,d^{-1}c^{-1}aba^{-1}b^{-1}c))d\omega.
\]
Using the bijection $G^E\simeq G^{\vert E\vert}=G^4$, and denoting respectively by $x_1,x_2,x_3,x_4$ the variables of integration corresponding to the edges $a,b,c,d$, one gets
\[
\mu_{\Gbb,G,\Sigma,\vol}(d\omega)=\frac{1}{Z_G(1,t)}p_{t_1}(x_4)p_{t_2}(x_4^{-1}x_3^{-1}x_1x_2x_1^{-1}x_2^{-1}x_3)dx_1dx_2dx_3dx_4.
\]
In particular,
\[
Z_G(1,t)=\int_{G^4}p_{t_1}(x_4)p_{t_2}(x_4^{-1}x_3^{-1}x_1x_2x_1^{-1}x_2^{-1}x_3)dx_1dx_2dx_3dx_4.
\]
Using the semigroup property of heat kernel yields
\[
Z_G(1,t)=\int_{G^3}p_t(x_3^{-1}x_1x_2x_1^{-1}x_2^{-1}x_3)dx_1dx_2dx_3,
\]
and as $p_t(gxg^{-1})=p_t(x)$ for all $x,g\in G$, one finally gets
\[
Z_G(1,t)=\int_{G^2}p_t([x_1,x_2])dx_1dx_2,
\]
where $[x,y]=xyx^{-1}y^{-1}$ is the commutator of elements of $G$.
\end{example}

\begin{remark}
In the literature, this definition of the Yang--Mills measure is sometimes described as a discrete gauge theory for the \emph{heat kernel action}, or \emph{Villain action}, as opposed to the \emph{Wilson action} defined by Wilson in \cite{Wil74} by
\begin{equation}
S(\omega)=\sum_{f\in F}\Re(\Tr(1-\Omega^\omega(f))).
\end{equation}
In the case of the Wilson action, the Yang--Mills measure with coupling constant $\beta>0$ is defined as the Gibbs measure of inverse temperature $\beta$ and Hamiltonian $S$, that is,
\begin{equation}
\mu(d\omega)=\frac{1}{Z}e^{-\beta S(\omega)}d\omega.
\end{equation}
The advantage of the Wilson action, mainly used in lattice gauge theory \cite{Cha20,CPS}, is that it defines a simpler model in terms of statistical mechanics, but going from the discrete to the continuous model seems more complicated in general and requires renormalization -- which is not the case for the heat kernel action, as we shall see in a moment. Note that some recent papers establish a correspondence between Villain gauge theory and Wilson gauge theory \cite{CG,SSZ}, respectively in $\Z^d$ and in $\R^2$. In the sequel, we shall always use the heat kernel action.
\end{remark}

\begin{remark}
The Yang--Mills measure can also be defined for a surface $\Sigma$ with boundary. A way to do it is to see $\Sigma$ as a subset of a surface $\Sigma_0$ without boundary but with marked simple loops whose union gives $\partial\Sigma$, and the map embedded in $\Sigma$ is a graph $\Gbb$ embedded in $\Sigma_0$ such that $\partial\Sigma$ is a concatenation of edges of $\Gbb$, and whose faces that do not belong to $\Sigma$ are marked faces that do not contribute to the definition of the measure. One can then impose boundary conditions, cf. \cite{Lev10,DL}. In particular, the measure on $\R^2$ can be defined as a measure on a planar map (i.e. a map embedded in the sphere) with a marked face, the ``unbounded face".
\end{remark}

A fundamental property of the Yang--Mills measure is its invariance by subdivision. It implies that the partition function does not depend on the graph, but only on the topology and area of the surface. Before we state it, let us introduce a notation: given two topological maps $\Gbb$ and $\Gbb'$, we say that $\Gbb$ is \emph{finer} than $\Gbb'$, and write $\Gbb'\preceq\Gbb$, if all edges of $\Gbb'$ correspond to concatenations of edges of $\Gbb$. In this case, any face of $\Gbb'$ is the reunion of faces of $\Gbb$ there is a natural restriction map $\mathcal{R}:G^{E}\to G^{E'}$, given by the following: any edge $e'\in E'$ can be written as a concatenation $e'=e_1\ldots e_n$ of edges of $E$, hence for $\omega\in G^{E'}$, we can write
\[
\mathcal{R}(\omega)(e')=\omega(e_1\ldots e_n)=\omega(e_1)\ldots\omega(e_n).
\]

\begin{theorem}
Let $\Gbb=(V,E,F)$ be a topological map embedded in a surface $\Sigma$ of genus $g\geq 0$, and $\Gbb'=(V,E',F')$ be a topological map obtained by removing one edge of $\Gbb$, endowed with the same embedding. If we denote by $\mathcal{R}:G^{E}\to G^{E'}$ the restriction map, then
\begin{equation}
(\mathcal{R})_*\mu_{\Gbb,G,\Sigma,\vol} = \mu_{\Gbb',G,\Sigma,\vol}.
\end{equation}
In particular, the partition function $Z_{\Gbb,G}(g,t)$ does not depend on $\Gbb$, and satisfies
\begin{itemize}
\item If $g\geq 1$,
\begin{equation}\label{eq:PartitionFunction2}
Z_{\Gbb,G}(g,t)=Z_G(g,t) = \int_{G^{2g}} p_t([x_1,y_1]\cdots [x_g,y_g])dx_1dy_1\cdots dx_g dy_g.
\end{equation}
\item If $g=0$,
\begin{equation}\label{eq:PartitionFunction2bis}
Z_{\Gbb,G}(0,t)=Z_G(0,t) = p_t(1).
\end{equation}
\end{itemize}
\end{theorem}

\begin{proof}
The proof consists into two steps. We first prove the result for the unnormalized measures $\tilde{\mu}_{\Gbb,G,\Sigma,\vol}$, then we show that the partition function does not depend on the graph. These two steps will imply the result for the normalized measures.

{\bf Step 1.} Let $e\in E$ be the edge removed from $\Gbb$ to get $\Gbb'$, so that $E=E'\cup\{e\}$. If we require that $\Gbb$ remains a map, it means that $e$ borders two faces $f_1,f_2$ of $\Gbb$, whose union is a single face $f_*\in F'$, see Figure~\ref{fig:face_split}. Set $F_0=F\setminus\{f_1,f_2\}$ and $F'_0=F'\setminus\{f_*\}$.
\begin{figure}[b!]
    \centering
    \includegraphics[width=\linewidth]{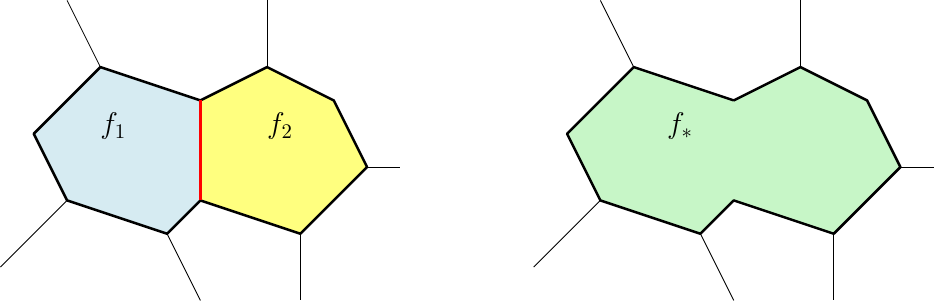}
    \caption{\small A portion of $\Gbb$ (on the left) and $\Gbb'$ (on the right). The faces $f_1$ and $f_2$ of $\Gbb$ are separated by the edge $e$ (in red). After removal of $e$, $f_1$ and $f_2$ are replaced by $f_*$ (on the right). All other faces and edges remain unchanged.}
    \label{fig:face_split}
\end{figure}
Let $\Phi:G^{E'}\to\R$ be a bounded measurable function. It corresponds to a function on $G^{E'}\sim G^{\vert E'\vert}$ invariant by the action of $G^{V'}$, that we shall also denote by $\Phi$. According to the Driver--Sengupta formula,
\[
\int_{G^{E'}} \Phi(\omega) \mathcal{R}_*\tilde{\mu}_{\Gbb,G,\Sigma,\vol}(d\omega) = \int_{G^E} \Phi\circ\mathcal{R}(\omega)\prod_{f\in F} p_{\vol(f)}(h_{\partial f})d\omega.
\]
Let us label arbitrarily the edges of $E'$ by $(e_1,\ldots,e_k)$, where $k=\vert E\vert$, so that $E=\{e_1,\ldots,e_k,e\}$. We get
\[
\int_{G^E} \Phi\circ\mathcal{R}(\omega) \tilde{\mu}_{\Gbb,G,\Sigma,\vol}(d\omega)=\int_{G^{k+1}}\Phi(e_1,\ldots,e_k)\prod_{f\in F}p_{\vol(f)}(h_{\partial f})de_1\ldots de_k de.
\]
Set
\[
P(e_1,\ldots,e_k)=\Phi(e_1,\ldots, e_k)\prod_{f\in F_0}p_{\vol(f)}(h_{\partial f}),
\]
and $t_1=\vol(f_1)$, $t_2=\vol(f_2)$. On the one hand,
\begin{equation}\label{eq:int_Cgg_1}
\int_{G^E} \Phi\circ\mathcal{R}(\omega) \tilde{\mu}_{\Gbb,G,\Sigma,\vol}(d\omega)=\int_{G^{k+1}}P(e_1,\ldots,e_k)p_{t_1}(h_{\partial f_1})p_{t_2}(h_{\partial f_2})de_1\ldots de_k de,
\end{equation}
and on the other hand
\begin{equation}\label{eq:int_Cgg_2}
\int_{G^{E'}} \Phi(\omega) \tilde{\mu}_{\Gbb',G,\Sigma,\vol}(d\omega)=\int_{G^{k}}P(e_1,\ldots,e_k)p_{t_1+t_2}(h_{\partial f*})de_1\ldots de_k.
\end{equation}
There are words $\alpha,\beta$ in the edges of $\Gbb'$ such that $\partial f_1=e\alpha$, $\partial f_2=\beta e^{-1}$ and $\partial f_*=\beta\alpha$. We deduce that $\partial f_*=\partial f_2\partial f_1$.

Let us integrate the RHS of~\eqref{eq:int_Cgg_1} with respect to $e$ and use the semigroup property of the heat kernel:
\begin{align*}
\int_{G^E} \Phi\circ\mathcal{R}(\omega) \tilde{\mu}_{\Gbb,G,\Sigma,\vol}(d\omega)= &\int_{G^k}P(e_1,\ldots,e_k)p_{t_1+t_2}(\partial f_2\partial f_1)p_{t_1}(e\alpha)de_1\ldots de_k\\
= & \int_{G^{E'}}\Phi(\omega) \tilde{\mu}_{\Gbb',G,\Sigma,\vol}(d\omega),
\end{align*}
which shows the invariance of the unnormalized measures.

{\bf Step 2.} The partition function $Z_{\Gbb,G}(g,t)$ is the integral of the unnormalized measure $\tilde{\mu}_{\Gbb,G,\Sigma,\vol}$, which is invariant by subdivision as previously shown. If $g\geq 1$, we can obtain~\eqref{eq:PartitionFunction2} by applying the Driver--Sengupta formula to a map of genus $g$ with one face. We can take as a map a $4g$-gon with boundary $[a_1,b_1]\ldots[a_g,b_g]$. Finally, if $g=0$, we apply the Driver--Sengupta formula to a planar map with two faces separated by a simple loop.
\end{proof}

\subsection{The group of loops}\label{sec:loop_group}

The main observables of Yang--Mills theory are \emph{Wilson loops} $W_\ell^\chi = \chi(h_\ell)$, where $\ell$ is a loop traced in $\Gbb$ and $\chi:G\to\C$ is a central function on $G$, that is
\[
\chi(k^{-1}gk)=\chi(g),\quad \forall g,k\in G.
\]
In practice, we shall only consider the following Wilson loops:
\[
W_\ell = \tr(h_\ell),
\]
where $\tr=\frac1N\Tr$ is the normalized trace on $\mathcal{M}_N(\C)$, because they form a dense generating family in the algebra of Wilson loops when $G$ is a compact classical group, according to a theorem of L\'evy \cite{Lev04}. They constitute random variables, whose expectation is either called \emph{Wilson loop expectation}, or \emph{vacuum expectation value} (VEV). A first remark is that Wilson loop expectations vanish for open paths.

\begin{proposition}
Let $\gamma$ be an open path in $\Gbb$. Then $\E[W_\gamma]=0$.
\end{proposition}

\begin{proof}
It is a straightforward consequence of gauge invariance.
\end{proof}

According to the previous proposition, it is sufficient to focus on Wilson loops for loops in the graph. We will see what structures can be put on the set of loops in the graph $\Gbb$. A first remark is that one can endow $\mathrm{P}(\Gbb)$ with an operation of multiplication, given by concatenation of non-empty words: if $\gamma=e_1\ldots e_m$ and $\gamma_2=e'_1\ldots e'_n$ and if $\overline{e_m}=\underline{e'_1}$, then $\gamma_1\cdot\gamma_2=\gamma_1\gamma_2=e_1\ldots e_m e'_1\ldots e'_n$. One can also extend the concatenation to constant words: $\gamma\cdot\overline{\gamma}=\underline{\gamma}\cdot\gamma = \gamma$ by convention. For any path $\gamma=e_1\cdots e_n\in\mathrm{P}(\Gbb)$, we define the reverse path $\gamma^{-1}=e_n^{-1}\cdots e_1^{-1}$. Even if we restrict ourselves to loops, there is no satisfying algebraic structure: $\mathrm{L}(\Gbb)$ is not stable by multiplication, and there is no neutral element for the inversion. A simple way to overcome the first issue is to fix a common basepoint for all loops. It might seem to be restrictive, but we will see in a moment that it is not, because of the way we deal with the second issue.

In order to get a neutral element, we put an equivalence class on loops as follows: $\ell_1$ and $\ell_2\in\mathrm{L}(\Gbb)$ are in the \emph{same reduction class} if one can go from one to the other by a finite number of insertion and removals of edges. More concretely,
\[
e_1\ldots e_i \tilde{e}\tilde{e}^{-1}e_{i+1}\ldots e_n \sim e_1\ldots e_i e_{i+1}\ldots e_n.
\]
Two equivalent loops in a triangulation of a surface of genus 2 are shown in Figure~\ref{fig:equiv_loops} (sides of the polygons are glued pairwise according to their labels in order to get a genus 2 surface).
\begin{figure}[t!]
    \centering
    \includegraphics[width=\linewidth]{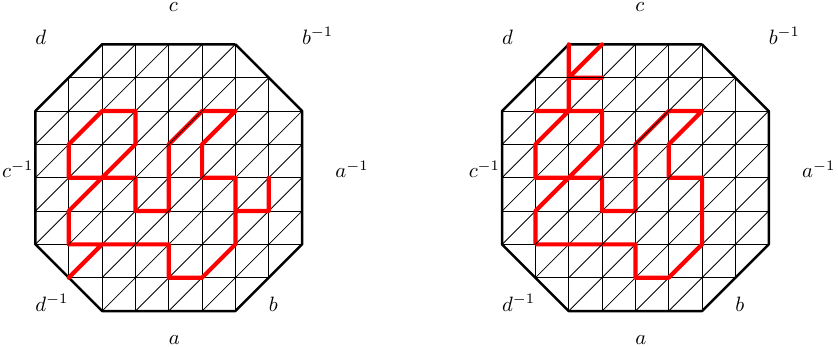}
    \caption{\small The loop on the left and on the right are in the same reduction class.}
    \label{fig:equiv_loops}
\end{figure}
This equivalence relation is related to the tree-like equivalence sometimes found in the literature \cite{GP,HL}. We shall denote by $\mathrm{RL}_v(\Gbb)=\mathrm{L}_v(\Gbb)/\sim$ the set of \emph{reduced loops}. The following result is simple, yet quite useful.

\begin{proposition}
For all $v\in V$,  $\mathrm{RL}_v(\Gbb)$ is a group.
\end{proposition}

We will state a few important results about this group, that come from \cite{Lev10}. But first, let us explain why fixing a basepoint is not a big deal for our purpose. Fix a map $\Gbb=(V,E,F)$ of any genus, and consider a marked vertex $v\in V$. If we want to compute the Wilson loop expectation of a loop $\ell$ that does not go through $v$, here is what one can do. Let $v_0$ the initial basepoint of $\ell$, and choose an oriented path $c$ in $\Gbb$ going from $v$ to $v_0$. Then the loop $\ell_v=c\ell c^{-1}$ has basepoint $v$ and is in the same reduction class as $\ell$. See Figure~\ref{fig:lasso} for an example.
\begin{figure}[b!]
    \centering
    \includegraphics[width=0.8\linewidth]{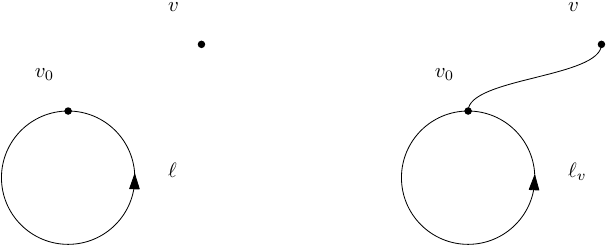}
    \caption{Construction of a loop $\ell_v$ in the same reduction class as $\ell$ but based at $v$.}
    \label{fig:lasso}
\end{figure}
In the case where $\ell=\partial f$ is the boundary of a face $f\in F$, then the loop $\ell_v$ is called \emph{lasso} of base $v$ associated to the face $f$. A priori, there are many possible $\ell_v$ for a given $\ell$, but if we fix a spanning tree $T$ of $\Gbb$ (i.e. a connected subgraph without cycles that goes through all vertices of $\Gbb$), there is a unique way to take a path from $v$ to the boundary of $f$ through $T$. Lassos are one of the main tools to prove the following two results.

\begin{lemma}[\cite{Lev10}, Lemma 1.3.33]
For any map $\Gbb=(V,E,F)$ of genus $g\geq 0$ and for any $v\in V$, the group $\mathrm{RL}_v(\Gbb)$ is free of rank $\vert E\vert - \vert V\vert +1=2g-1+\vert F\vert$.
\end{lemma}

\begin{proof}
Let us take a spanning tree $T=(V,E_T)$ of $\Gbb$. For any $e\in E\setminus E_T$, let $\ell_e=[v,\underline{e}]e[\overline{e},v]$ be a loop obtained by taking a path $[v,\underline{e}]$ from $v$ to the starting point of $e$ through $T$, then going through $e$, and taking a path $[\overline{e},v]$ from the endpoint of $e$ to $v$ through $T$ only. By construction, for any $e\in E\setminus E_T$ the loop $\ell_e$ is uniquely defined and is a lasso. It is not hard to see that $(\ell_e,e\in E\setminus E_T)$ generate $\mathrm{RL}_v(\Gbb)$, since any loop $\ell=e_1\ldots e_n$ is in the same equivalence class as $\ell_{e_1}\ldots \ell_{e_n}$, and one can also check that it is free. Given that $\vert E\setminus E_T\vert=\vert E\vert-(\vert V\vert -1)$, we get that $\mathrm{RL}_v(\Gbb)$ has rank $\vert E\vert-\vert V\vert+1$ as expected.
\end{proof}

\begin{proposition}[\cite{Lev10}, Prop. 2.4.2]\label{prop:tame_generators}
Let $\Gbb=(V,E,F)$ be a map or genus $g$ and $v\in V$ be a fixed vertex. Set $k=\vert F\vert$. Then:
\begin{enumerate}
\item The group $\mathrm{RL}_v(\Gbb)$ has the following presentation
\[
\langle a_1,b_1,\ldots,a_g,b_g,c_1,\ldots,c_k\vert [a_1,b_1]\ldots[a_g,b_g]=c_1\ldots c_k\rangle,
\]
where the homotopy classes of $(a_1,b_1,\ldots,a_g,b_g)$ generate the fundamental group of $\Sigma$, and $c_1,\ldots,c_k$ are lassos enclosing the faces of $\Gbb$.
\item For any bounded measurable function $f:G^{2g+k}\to\C$,
\begin{align*}
Z_G(g,t)&\int_{G^{2g+k}}f(H_{a_1},H_{b_1},\ldots,H_{a_g},H_{b_g},H_{c_1},\ldots,H_{c_k})\mu_{\Gbb,G,\Sigma,\vol}(d\omega)\\
= & \int_{G^{2g+k-1}}f(x_1,y_1,\ldots,x_g,y_g,z_1,\ldots,z_k)\prod_{i=1}^kp_{\vol(f_i)}(z_i)\prod_{i=1}^gdx_idy_i\prod_{j=1}^{k-1}dz_j,
\end{align*}
by setting $z_k=z_{k-1}^{-1}\ldots z_1^{-1}[x_1,y_1]\ldots[x_g,y_g]$, and with the convention that $H_\ell=H_\ell(\omega)$ as a random variable for any loop $\ell$.
\end{enumerate}
\end{proposition}

Proposition~\ref{prop:tame_generators} yields another interpretation of the discrete Yang--Mills measure: it is a measure on the \emph{representation space}
\[
\Hom(\mathrm{RL}_v(\Gbb),G)
\]
which is invariant by the simultaneous adjoint action of $G$ on all generators. It descends to a measure on
\[
\Hom(\mathrm{RL}_v(\Gbb),G)/G,
\]
which is actually the same space as the \emph{character variety}
\[
\chi(\Sigma_F,G)=\Hom(\pi_1(\Sigma_F),G)/G,
\]
for the surface $\Sigma_F$ obtained from $\Sigma$ by removing one point from the interior of each face $f\in F$. Indeed, the groups $\mathrm{RL}_v(\Gbb)$ and $\pi_1(\Sigma_F)$ have the same presentation, therefore they are isomorphic. The character variety $\chi(\Sigma_F,G)$ has a rich structure that we will not develop here; we refer to \cite{Lab} for an introduction.

\subsection{Continuous extension}\label{sec:continuous_extension}

The property of invariance by subdivision satisfied by the discrete Yang--Mills theory enables to define it for graphs that are finer and finer, and one could expect to extend it into a ``continuous" measure by taking a projective limit of graphs. The underlying argument is the same as Kolmogorov's extension theorem to prove the existence of the Brownian motion from its finite-dimensional marginals. The moral is that one does not lose any generality to use the discrete Yang--Mills measure, as long as we consider observables associated to fixed loops: under loose regularity assumptions, they can always be drawn in some map and all computations can be done in the discrete setting.

To define the continuous limiting object, we need to choose a correct space of paths in the surface that contains at least all paths that can be traced in maps.

\begin{definition}
Let $\Sigma$ be a smooth compact surface. A \emph{parametrized path} on $\Sigma$ is a path $\gamma:[0,1]\to\Sigma$ that is either constant, or Lipschitz-continuous, with a speed bounded from below by a positive constant. A \emph{path} on $\Sigma$ is an equivalence class of parametrized paths on $\Sigma$ for the equivalence relation
\[
\gamma\sim\gamma'\Leftrightarrow \exists \Phi:[0,1]\to[0,1] \text{ bi-Lipschitz increasing homeo. s.t. } \gamma = \gamma'\circ\Phi.
\]
We denote by $\mathrm{P}(\Sigma)$ the set of paths in $\Sigma$, and $\mathrm{L}(\Sigma)$ the set of \emph{loops} in $\Sigma$, i.e. closed paths.
\end{definition}

The reason why we require non-constant paths to be Lipschitz-continuous with speed bounded from below is that it enables the choice of a parametrization with constant speed. The main result of this subsection, due to L\'evy \cite{Lev10}, is the following.

\begin{theorem}\label{thm:extension_Levy}
Let $\Sigma$ be a compact surface endowed with a smooth measure of area $\vol$, and $G$ a compact Lie group whose Lie algebra is endowed with an invariant inner product. There exists a family of $G$-valued random variables $(H_\ell)_{\ell\in\mathrm{L}(\Sigma)}$ such that:
\begin{enumerate}
\item For any map $\Gbb=(V,E,F)$ embedded in $\Sigma$ and any fixed $v\in V$, the distribution of $(H_\ell)_{\ell\in\mathrm{RL}_v(\Gbb)}$ is given by Proposition~\ref{prop:tame_generators}~(ii).
\item For any family $(\ell_n)_{n\geq 1}$ of loops that converges\footnote{In the sense of uniform convergence of paths with their constant speed parametrization.} to a loop $\ell$, then the sequence of random variables $(H_{\ell_n})_{n\geq 1}$ converges in probability to $\ell$.
\end{enumerate}
The \emph{continuous Yang--Mills measure} on $\Sigma$ is the distribution of $(H_\ell)_{\ell\in\mathrm{L}(\Sigma)}$.
\end{theorem}

The proof of this result relies on many technical developments in measure theory and Riemannian geometry, that we will skip here. Instead, we will rather make a few remarks:

\begin{itemize}
\item The right point of view to get a projective system of measures is to consider graphs whose edges are embedded into piecewise geodesic paths, which requires fixing a Riemannian metric on $\Sigma$.
\item However, one can show that all discrete Yang--Mills measures are invariant by area-preserving diffeomorphisms, therefore they do not depend on the choice of the Riemannian metric.
\item The choice of regularity of the loops comes from the fact that they can be obtained as limits of piecewise geodesic loops. It is also true for rectifiable paths, which are continuous with finite length. Note that having finite length does not depend on the choice of a metric, which justifies point (2) in Theorem~\ref{thm:extension_Levy}.
\end{itemize}

To conclude this section, we must clarify an important subtlety regarding the continuous Yang--Mills measure. In the construction of the discrete Yang--Mills measure, we mentioned that discrete connections were implicitly identified with their holonomies along loops, which was quite acceptable. It becomes much less obvious at the continuous level, and the convoluted construction of the continuous Yang--Mills measure can leave a doubt whether it is indeed related to the formal Yang--Mills measure~\eqref{eq:YM_heuristique_eucl}. A first step was achieved by L\'evy and Norris \cite{LN06}, who proved a large deviation principle of the Yang--Mills measure on a compact surface with area $t$ when $t\to 0$, whose rate function is proportional to the Yang--Mills action $S_\YM$. Replacing random connections by random holonomies is here harmless because the observables of the theory are essentially Wilson loops, which are central functions of the holonomies. This point of view has been used with a relative success in an attempt of defining quantum gravity, and it is called the \emph{loop representation}, see \cite{GP}. Another attempt to describe the Yang--Mills measure directly in terms of random connections has been made by Chevyrev \cite{Che19} on the torus, and it is generalized in forthcoming papers by other authors \cite{BCDRT,DangNohra}. In all these works, the holonomies of random connections along \emph{suitable} loops (which are slightly less general than the ones considered in \cite{Lev10} and depend notably on some gauge-fixing procedure) are proved to be identical in distribution to the marginals of the Yang--Mills holonomy process. The measure on connections on the torus defined in \cite{Che19} has been also proved to be the universal limit of the 2D Yang--Mills Langevin dynamic in a series of two (long and technical) papers \cite{CCHS22,CheShe23}.

\section{Character expansion of the partition function}\label{sec:AH_groupes}

In this section, we will recall the main results of representation theory of compact groups required to compute the Fourier transform of the partition function, which will be stated in Theorem~\ref{thm:fourier_pf}. We shall mainly refer to \cite{BtD,Far}.

\subsection{Representations of Lie groups}

\begin{definition}
Let $G$ be a Lie group. A real (resp. complex) \emph{representation} of $G$ is a couple $(\rho,V)$ where $V$ is a real (resp. complex) vector space and $\rho:G\to\GL(V)$ is a smooth group morphism. The \emph{degree} of $\rho$, also called its \emph{dimension}, is the dimension of the vector space $V$ and denoted by $d_\rho$.
\end{definition}

A representation $(\rho,V)$ of $G$ is said to be \emph{irreducible} if the only subspaces of $V$ stable by $\rho(g)$ for any $g\in G$ are $\{0\}$ and $V$. Irreducible representations somehow constitute the building blocks of representation theory.

Note that for any finite-dimensional representation $(\rho,\mathcal{H})$ of a compact Lie group $G$ in a Hilbert space\footnote{Here, it might be overkill to speak about a Hilbert space, since it is a finite-dimensional vector space, but we prefer to use this denomination to underline the fact that it is endowed with an inner product. In general, there might be unitary representations in infinite-dimensional Hilbert spaces, e.g. the left and right regular representations act on $L^2(G)$.} $\mathcal{H}$, there is an inner product on $\mathcal{H}$ such that the representation is unitary, i.e.
\[
\Vert \pi(g)v\Vert = \Vert v\Vert,\quad \forall g\in G,\ \forall v\in\mathcal{H}.
\]
Indeed, if we start from an arbitrary inner product $\langle\cdot,\cdot\rangle_0$, the new inner product $\langle\cdot,\cdot\rangle$ defined by
\[
\langle v,w\rangle = \int_G \langle \rho(g)v,\rho(g)w\rangle_0 dg,\quad \forall (v,w)\in \mathcal{H}^2,
\]
turns $\rho$ into a unitary representation.

\begin{lemma}[Schur]\label{lem:schur}
Let $(\rho_1,V_1)$ and $(\rho_2,V_2)$ be two irreducible complex representations of a Lie group $G$, and $A\in\Hom(V_1,V_2)$ be a linear application such that $A\rho_1=\rho_2 A$ (it is said to be $G$-\emph{equivariant}). Then one of the following assertions is true:
\begin{enumerate}
\item $A = 0$,
\item $A$ is an isomorphism.
\end{enumerate}
Furthermore, if $V_1=V_2$ and $\rho_1=\rho_2$, then there exists $\lambda\in\C$ such that $A=\lambda I_{V_1}$.
\end{lemma}

Schur's Lemma enables to define an equivalence relation on irreducible representations, and we denote by $\widehat{G}$ the set of equivalence classes, sometimes called the \emph{dual} of $G$. Note that if $G$ is compact, then $\widehat{G}$ is countable, and if $G$ is abelian then $\widehat{G}$ is also a group.

\subsection{Peter--Weyl theorem}

Let $G$ be a compact group and $(\rho,\mathcal{H})$ be a unitary irreducible representation of $G$. The \emph{matrix coefficients} of $\rho$ are functions of the form
\[
g\mapsto \langle \rho(g) u, v\rangle,
\]
for $u,v\in\mathcal{H}$. The space of all matrix coefficients is denoted by $\mathscr{M}_\rho$, and if $\{e_1,\ldots,e_k\}$ is an orthonormal basis of $\mathcal{H}$, $\mathscr{M}_\rho$ is generated by the \emph{elementary matrix coefficients}
\[
\rho_{ij}:g\mapsto \langle\rho(g)e_i,e_j\rangle,\quad \forall 1\leq i,j\leq k.
\]

\begin{theorem}[Schur's orthogonality relations I]\label{thm:orth_schur1}
Let $(\rho,\mathcal{H})$ be a unitary irreducible representation of a compact Lie group $G$. For any $u,v,u',v'\in\mathcal{H}$,
\begin{equation}\label{eq:orth_schur1}
\int_G \langle\rho(g)u,v\rangle\overline{\langle \rho(g)u',v'\rangle} dg = \frac{1}{d_\rho}\langle u,u'\rangle \langle v,v'\rangle.
\end{equation}
In particular, if $\{e_1,\ldots,e_{d_\rho}\}$ is an orthonormal basis of $\mathcal{H}$, then
\begin{equation}\label{eq:orth_schur2}
\int_G \rho_{ij}(g)\overline{\rho_{k\ell}(g)}dg = \frac{1}{d_\rho}\delta_{ik}\delta_{j\ell}.
\end{equation}
\end{theorem}

\begin{theorem}[Schur's orthogonality relations II]\label{thm:orth_schur2}
Let $(\rho,\mathcal{H})$ and $(\rho',\mathcal{H}')$ be two irreducible unitary representations of a compact Lie group $G$. If $\rho$ and $\rho'$ are not equivalent, then $\mathscr{M}_\rho$ and $\mathscr{M}_{\rho'}$ are orthogonal in $L^2(G)$.
\end{theorem}

A fundamental consequence of these orthogonality relations is the following theorem:

\begin{theorem}[Peter--Weyl]\label{thm:Peter-Weyl}
Let $G$ be a compact Lie group. For any $\lambda\in\widehat{G}$, denote by $\mathscr{M}_\lambda$ the space of matrix coefficients of a representation of the class $\lambda$. We have
\begin{equation}
L^2(G) = \widehat{\bigoplus_{\lambda\in\widehat{G}}}\ \mathscr{M}_\lambda,
\end{equation}
where $\widehat{\bigoplus}$ denotes the $L^2$ closure of the algebraic direct sum.
\end{theorem}

One of the main applications of this theorem is Plancherel's formula for compact groups, that we shall see in the next subsection. Beforehand, let us rewrite the Peter--Weyl theorem under a more specific form. A function $f:G\to\C$ on a compact Lie group is said to be \emph{central} if for all $(g,h)\in G^2$, $f(ghg^{-1})=f(h)$. It is in particular the case for the \emph{characters}
\[
\chi_\rho:g\mapsto \Tr(\rho(g)),
\]
for any finite-dimensional representation $\rho$ of $G$. By convention, if $\rho$ is an irreducible representation of the class $\lambda\in\widehat{G}$, we will denote by $\chi_\lambda$ the character associated to $\rho$. In particular, $\chi_\lambda\in\mathscr{M}_\lambda$.

\begin{theorem}[Peter--Weyl II]\label{thm:Peter-Weyl2}
For any compact Lie group $G$, the set $\{\chi_\lambda,\lambda\in\widehat{G}\}$ is a Hilbert basis of the space $Z^2(G)\subset L^2(G)$ of square integrable central functions.
\end{theorem}

\subsection{Fourier transform}

\begin{definition}
The \emph{Fourier transform} of $f\in L^2(G)$ is the function $\hat{f}:\widehat{G}\to\End(\mathcal{H}_\lambda)$ defined by
\[
\hat{f}(\lambda) = \int_G f(g)\rho_\lambda(g^{-1})dg,
\]
where $(\rho_\lambda,\mathcal{H}_\lambda)$ is an irreducible unitary representation of $G$ in the class $\lambda$.
\end{definition}

\begin{theorem}[Plancherel]
The map $f\mapsto\hat{f}$ is a unitary isomorphism of Hilbert spaces between $L^2(G)$ and the space of sequences $A=(A_\lambda)_{\lambda\in\widehat{G}}$ of operators $A_\lambda\in\End(\mathcal{H}_\lambda)$ such that
\[
\Vert A\Vert^2=\sum_{\lambda\in\widehat{G}}d_\lambda \Tr(A_\lambda^*A_\lambda)<\infty.
\]
Moreover, for any $f\in L^2(G)$,
\begin{equation}\label{eq:decomp_L2_f}
f(g) = \sum_{\lambda\in\widehat{G}}d_\lambda \Tr(\hat{f}(\lambda)\rho_\lambda(g)).
\end{equation}
\end{theorem}

One can give an interpretation of~\eqref{eq:decomp_L2_f} in terms of convolution:

\begin{corollary}
Let $f\in L^2(G)$ for a compact Lie group $G$. The following identity holds in $L^2(G)$:
\begin{equation}
f(g) = \sum_{\lambda\in\widehat{G}} d_\lambda (f*\chi_\lambda)(g).
\end{equation}
\end{corollary}

\begin{proof}
For any $\lambda\in\widehat{G}$, take any representation $(\rho_\lambda,\mathcal{H}_\lambda)$. By~\eqref{eq:decomp_L2_f} and the definition of Fourier coefficients,
\[
f(g) = \sum_{\lambda} d_\lambda \Tr\left(\int_G f(h)\rho_\lambda(h^{-1})dh\rho_\lambda(g)\right),
\]
and by the linearity of the trace,
\[
f(g) = \sum_{\lambda} d_\lambda \int_G f(h)\Tr(\rho_\lambda(h^{-1})\rho_\lambda(g))dh.
\]
The result follows then from straightforward simplifications.
\end{proof}

\subsection{Character expansion of the heat kernel}

Recall that the heat kernel on $G$ is the convolution semigroup $(p_t)_{t\geq 0}$ that satisfies the heat equation
\begin{equation}
\left\lbrace\begin{array}{ccl}
\frac{d}{dt}p_t(g) & = & \frac12\Delta_G p_t(g),\quad \forall g\in G, \forall t>0,\\
\lim_{t\downarrow 0}p_t & = & \delta_{1_G}.
\end{array}\right.
\end{equation}
It clearly depends on the choice of the inner product on $\mathfrak{g}$. As in the case of heat equation in $\R^d$, one can get an explicit formula for the heat kernel in terms of eigenvalues and eigenfunctions of the Laplacian by the means of Fourier transform. In our case, the eigenfunctions of $\Delta_G$ are in fact the irreducible characters.

\begin{proposition}
Let $G$ be a compact Lie group. The set $\{\chi_\lambda,\lambda\in\widehat{G}\}$ is a Hilbert basis of eigenvectors of $\Delta_G$. Let us denote by $c_2(\lambda)\geq 0$ the eigenvalue of $-\Delta_G$ associated to $\chi_\lambda$.
\end{proposition}

The main result of this subsection is the following.

\begin{theorem}\label{thm:HK_decomp}
Let $G$ be a compact Lie group. The heat kernel $(p_t)_{t\geq 0}$ admits the following decomposition, which holds pointwise for $t>0$ and in $L^2(G)$ for all $t\geq 0$:
\begin{equation}
p_t(g) = \sum_{\lambda\in\widehat{G}} e^{-\frac{t}{2}c_2(\lambda)} d_\lambda \chi_\lambda(g),\quad \forall g\in G. 
\end{equation}
\end{theorem}

\begin{proof}
By Plancherel's formula, we have for any $t>0$ and $g\in G$
\[
p_t(g) = \sum_{\lambda\in\widehat{G}}d_\lambda (p_t*\chi_\lambda)(g).
\]
If we use the linearity of the integral and apply the heat equation,
\[
\frac{d}{dt}(p_t*\chi_\lambda)(g) = \left(\left(\frac{d}{dt}p_t\right)*\chi_\lambda\right)(g) = \frac12 ((\Delta_G p_t)*\chi_\lambda)(g).
\]
Furthermore, integrating by parts yields
\[
(\Delta_G p_t)*\chi_\lambda = (p_t * \Delta_G\chi_\lambda),
\]
thus an application of the identity $\Delta_G\chi_\lambda=-c_2(\lambda)\chi_\lambda$ gives
\[
\frac{d}{dt} (p_t*\chi_\lambda)(g) = -\frac12 c_2(\lambda) (p_t*\chi_\lambda)(g),\quad \forall g\in G.
\]
We can conclude by solving the differential equation for $\lambda\in\widehat{G}$.
\end{proof}

\begin{example}
In the abelian case $G=\U(1)$, one recovers a classical expression of the heat kernel on the circle, or equivalently on $[0,2\pi]$ with periodic boundary conditions: recall that the irreducible characters of $\U(1)$ are $\chi_n:e^{i\theta}\mapsto e^{i n\theta},$ for any $n\in\Z$, and one can check that $\Delta \chi_n=-n^2\chi_n$, thus Theorem~\ref{thm:HK_decomp} yields
\[
p_t(e^{i\theta})=\sum_{n\in\Z} e^{-\frac{t}{2}n^2+in\theta}.
\]
\end{example}

\subsection{Exact formula for the partition function}

We have seen that the partition function $Z_G(g,t)$ can be written as an integral over $G^{2g}$ involving the heat kernel $p_t$, for $t>0$. The goal of this subsection is to obtain the following character expansion of this matrix integral, which was first found by Rusakov \cite{Rus2} based on Migdal's work \cite{Mig}, and almost simultaneously rediscovered by Witten \cite{Wit91}.

\begin{theorem}\label{thm:fourier_pf}
Let $G$ be a compact Lie group, $\Sigma_g$ a surface of genus $g\geq 0$ and area $t> 0$. The Yang--Mills partition function satisfies
\begin{equation}
Z_G(g,t) = \sum_{\lambda\in\widehat{G}}e^{-\frac{t}{2}c_2(\lambda)}d_\lambda^{2-2g}.
\end{equation}
\end{theorem}

A first step towards the proof of this result is to decompose the partition function in a series involving integrals over irreducible characters.

\begin{lemma}\label{lem:fourier_pf}
Let $G$ be a compact Lie group. For any $g\geq 1,$
\begin{equation}
Z_G(g,t) = \sum_{\lambda\in\widehat{G}}e^{-\frac{t}{2}c_2(\lambda)}d_\lambda\int_{G^{2g}}\chi_\lambda([x_1,y_1]\cdots[x_g,y_g])\prod_{i=1}^g dx_idy_i.
\end{equation}
\end{lemma}

\begin{proof}
It is a straightforward consequence of~\eqref{eq:PartitionFunction2} and Theorem~\ref{thm:HK_decomp}.
\end{proof}

It remains to understand how to compute the integral of an irreducible character applied to a product of commutators of uniform random variables. To do so, we will need the following two propositions.

\begin{proposition}\label{prop:orbite_adjointe}
Let $G$ be a compact group, and $\lambda\in\widehat{G}$ an equivalence class of irreducible representations. For any $x,y\in G$,
\begin{equation}
\int_G \chi_\lambda(xgyg^{-1})dg = \frac{\chi_\lambda(x)\chi_\lambda(y)}{d_\lambda}.
\end{equation}
\end{proposition}

\begin{proof}
Let $(\rho_\lambda,\mathcal{H}_\lambda)$ be a unitary representation in the class $\lambda$. For any $y\in G$, let $A_y\in\mathrm{End}(\mathcal{H}_\lambda)$ be the operator defined by
\[
A_y=\int_G\rho_\lambda(g)\rho_\lambda(y)\rho_\lambda(g^{-1})dg=\int_G\rho_\lambda(gyg^{-1})dg.
\]
For any $x\in G$, $A_y$ commutes with $\rho_\lambda(x)$:
\[
\rho_\lambda(x)A_y =\int_G \rho_\lambda(xgyg^{-1})dg = \int_G \rho_\lambda(hyh^{-1}x)dh = A_y\rho_\lambda(x).
\]
In this sequence of equalities, we have successively used the fact that the representation is a group morphism, then a change of variable $h=xg$ which yields $g^{-1}=h^{-1}x$. By Schur's lemma (Lemma~\ref{lem:schur}), there exists $C_\lambda(y)$ such that $A_y=C_\lambda(y)I_{\mathcal{H}_\lambda}$. Let us first compute $\Tr(A_y)=C_\lambda(y)d_\lambda$:
\[
\Tr(A_y)=\int_G \chi_\lambda(gyg^{-1})dg=\chi_\lambda(y),
\]
which yields $C_\lambda(y)=\chi_\lambda(y)/d_\lambda$. Then, let us compute $\Tr(\rho_\lambda(x)A_y))=\int_G\chi_\lambda(xgyg^{-1})dg$:
\[
\Tr(\rho_\lambda(x)A_y)=C_\lambda(y)\Tr(\rho_\lambda(x))=C_\lambda(y)\chi_\lambda(x)=\frac{\chi_\lambda(x)\chi_\lambda(y)}{d_\lambda}.
\]
We have finally proved the proposition.
\end{proof}

\begin{proposition}
Let $G$ be a compact group and $\lambda,\mu\in\widehat{G}$ two equivalence classes of irreducible representations.
\begin{equation}
\chi_\lambda*\chi_\mu = \frac{\chi_\lambda}{d_\lambda}\delta_{\lambda,\mu}.
\end{equation}
\end{proposition}

\begin{proof}
Let $(\rho_\lambda,\mathcal{H}_\lambda)$ be a unitary representation in the class $\lambda$. Denote by $\rho_\lambda^{ij}$ the matrix coefficients of the representations, i.e.
\[
\rho_\lambda^{ij}(x)=\langle \rho_\lambda(x) e_i,e_j\rangle,\quad \forall 1\leq i,j\leq d_\lambda.
\]
We have
\[
\chi_\lambda(xy) = \Tr(\rho_\lambda(xy))=\Tr(\rho_\lambda(x)\rho_\lambda(y)) = \sum_{i,j=1}^{d_\lambda} \rho_\lambda^{ij}(x)\rho_\lambda^{ji}(y).
\]
Using the same trick for a representation $\rho_\mu$ of the class $\mu$, we get
\[
\chi_\lambda(xy)\chi_\mu(y^{-1}) = \chi_\lambda(xy)\overline{\chi_\mu(y)} = \sum_{i=1}^{d_\lambda}\sum_{j=1}^{d_\lambda}\sum_{k=1}^{d_\mu} \rho_\lambda^{ij}(x)\rho_\lambda^{ji}(y)\overline{\rho_\mu^{kk}(y)}.
\]
Integrating with respect to $y$ yields
\[
\sum_{i,j,k}\rho_\lambda^{ij}(x)\int_G\rho_\lambda^{ji}(y)\rho_\mu^{kk}(y^{-1})dy=\sum_{i,j,k}\rho_\lambda^{ij}(x)\int_G\rho_\lambda^{ji}(y)\overline{\rho_\mu^{kk}(y)}dy,
\]
and Schur's orthogonality relation yields
\[
\chi_\lambda*\chi_\mu(x) = \frac{1}{d_\lambda} \sum_{i=1}^{d_\lambda}\rho_\lambda^{ii}(x)\delta_{\lambda,\mu},
\]
which is the expected result.
\end{proof}

\begin{remark}
The previous proposition gives the structure constants of $Z^2(G)$ as a convolution algebra, if we recall that it is the $L^2$-closure of the convolution algebra generated by the irreducible characters.
\end{remark}

The main integration result that we will need is the following.

\begin{proposition}\label{prop:int_commu}
Let $G$ be a compact group, and $\lambda\in\widehat{G}$ be an equivalence class of irreducible representations. For any $g\geq 1,$
\begin{equation}
\int_{G^{2g}}\chi_\lambda([x_1,y_1]\cdots[x_g,y_g])\prod_{i=1}^gdx_idy_i = d_\lambda^{1-2g}.
\end{equation}
\end{proposition}

\begin{proof}
We proceed by induction on $g$. The case $g=1$ is a consequence of Proposition~\ref{prop:orbite_adjointe} applied to $x=x_1$ and $y=x_1^{-1}$:
\[
\int_{G^2}\chi_\lambda(x_1y_1x_1^{-1}y_1^{-1})dx_1dy_1 = \int_G \frac{\chi_\lambda(x_1)\chi_\lambda(x_1^{-1})}{d_\lambda} = 1,
\]
where the last equality follows from orthogonality relations. Assume that the equality holds for a fixed $g$. Then
\begin{align*}
\int_{G^{2g+2}}\chi_\lambda([x_1,y_1]&\cdots[x_g,y_g][x,y])\prod_{i=1}^gdx_idy_i dxdy =\\
&\frac{1}{d_\lambda} \int_{G^{2g+1}} \chi_\lambda([x_1,y_1]\cdots[x_g,y_g]x)\chi_\lambda(x^{-1})\prod_{i=1}^g dx_idy_idx.
\end{align*}
The RHS can be rewritten
\begin{align*}
\frac{1}{d_\lambda}\int_{G^{2g}}(\chi_\lambda*\chi_\lambda)&([x_1,y_1]\cdots[x_g,y_g])\prod_{i=1}^g dx_idy_i=\\
& \frac{1}{d_\lambda^2}\int_{G^{2g}}\chi_\lambda([x_1,y_1]\cdots[x_g,y_g])\prod_{i=1}^g dx_idy_i,
\end{align*}
by using the formula of convolution of characters.
\end{proof}

\begin{proof}[Proof of Theorem~\ref{thm:fourier_pf}]
We simply combine Lemma~\ref{lem:fourier_pf} and Proposition~\ref{prop:int_commu}.
\end{proof}

\section{Large-$N$ limit of the partition function}\label{sec:SchurWeyl}

We can now specialize to the case of the unitary group
\[
G=\U(N)=\{U\in\GL(N,\C): U^*=U^{-1}\},
\]
whose Lie algebra is the space of skew-Hermitian matrices
\[
\mathfrak{u}(N)=\{X\in\mathcal{M}_N(\C): X^*=-X\}=i\mathcal{H}_N,
\]
where $\mathcal{H}_N=\{X\in\mathcal{M}_N(\C): X^*=X\}$ is the space of Hermitian matrices, on which is notably defined the Gaussian Unitary Ensemble. We will endow $\mathfrak{u}(N)$ with the following inner product:
\[
\langle X,Y\rangle = N\Tr(XY^*) = -N\Tr(XY).
\]
We will first recall the expressions of relevant quantities, such as the dimension and Casimir of irreducible representations, and then we will provide results about the asymptotics of the partition function.

\subsection{Highest weight theory for $\U(N)$}

In this subsection, we gather standard results of the representation theory of unitary groups in order to get explicit formulas for the character, dimension and Casimir of all irreducible representations. Good references are \cite{BtD,Far} for instance. Note that we need to specify an inner product on $\mathfrak{u}(N)$ in order to give an expression of the Casimir, because it is an eigenvalue of $-\Delta_{\U(N)}$, which depends on a choice of a metric on $G$.

\begin{theorem}[Weyl's character formula]\label{thm:Weyl_char}
The irreducible representations of $\U(N)$ (resp. $\SU(N)$) are in bijection with the $N$-tuples $(\lambda_1,\ldots,\lambda_N)\in\Z^N$ such that $\lambda_1\geq\ldots\geq\lambda_N$ (resp. $\lambda_1\geq\ldots\geq\lambda_N=0$), called \emph{highest weights}. Moreover, for any irreducible representation of highest weight $\lambda$ and any $U\in\U(N)$ or $\SU(N)$ with eigenvalues $(x_1,\ldots,x_N)$,
\begin{equation}
\chi_\lambda(U)=s_\lambda(x_1,\ldots,x_N),
\end{equation}
where $s_\lambda$ is the \emph{Schur function} associated to $\lambda$:
\[
s_\lambda(x_1,\ldots,x_N)=\frac{\det\left[x_i^{\lambda_j+N-j}\right]_{1\leq i,j\leq N}}{\det\left[x_i^{N-j}\right]_{1\leq i,j\leq N}}.
\]
\end{theorem}

An example of highest weight of $\U(N)$ is given in Figure~\ref{fig:highest_weight}.

\begin{figure}[b!]
    \centering
    \includegraphics[width=0.5\linewidth]{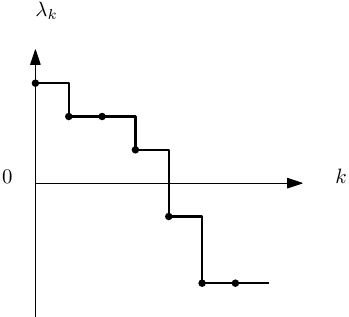}
    \caption{The highest weight $\lambda=(3,2,2,1,-1,-3,-3)$ of $\U(7)$.}
    \label{fig:highest_weight}
\end{figure}

\begin{remark}
These Schur functions are symmetric Laurent polynomials for $\U(N)$ and symmetric polynomials for $\SU(N)$.
\end{remark}

\begin{corollary}[Weyl's dimension formula]\label{thm:Weyl_dim}
Let $\lambda\in\widehat{\U}(N)$ (or $\widehat{\SU}(N)$). Its dimension satisfies
\begin{equation}
d_\lambda=\prod_{1\leq i<j\leq N}\frac{\lambda_i-\lambda_j+j-i}{j-i}.
\end{equation}
\end{corollary}

\begin{theorem}
Assume that $\mathfrak{u}(N)$ is endowed with the inner product $\langle X,Y\rangle=N\Tr(X^*Y)$. For any $\lambda\in\widehat{\U}(N)$,
\begin{equation}
c_2(\lambda)=\frac{1}{N}\left(\sum_{i=1}^N\lambda_i^2 + \sum_{1\leq i<j\leq N}(\lambda_i-\lambda_j)\right).
\end{equation}
\end{theorem}

Although we will not need it here, the Casimirs of irreducible representations of $\SU(N)$ are slightly different:
\[
c_2(\lambda)=\frac{1}{N}\left(\sum_{i=1}^N\lambda_i^2 + \sum_{1\leq i<j\leq N}(\lambda_i-\lambda_j)\right)-\frac{1}{N^2}\left(\sum_{i=1}^N\lambda_i\right)^2,\quad \forall \lambda\in\widehat{\SU}(N).
\]

\subsection{Genus $\geq 2$: Witten zeta functions}

If $G$ is a compact connected simply connected group, we define its \emph{Witten zeta function} as the meromorphic continuation of the series
\[
\zeta_G(s)=\sum_{\lambda\in\widehat{G}}d_\lambda^{-s},
\]
for all $s\in\C$. This name has been chosen by Zagier \cite{Zag92} because the function appears in the work of Witten \cite{Wit91}, and because it generalizes the Riemann zeta function: indeed, for any $s\in\C$,
\[
\zeta_{\SU(2)}(s)=\sum_{n\geq 1}\frac{1}{n^s}=\zeta(s).
\]
For the sake of simplicity, we will take it for granted that, for any $s>1$, 
\[
\sup_N\zeta_{\SU(N)}(s)<\infty,
\]
(see \cite[Prop. 2.2]{Lem} for a proof). The main result we care about is the following:

\begin{lemma}\label{lem:Zeta}
For all $s>1$,
\begin{equation}
\lim_{N\to\infty}\left\vert \zeta_{\SU(N)}(s)-1\right\vert = 0.
\end{equation}
\end{lemma}

\begin{proof}
Let us start by using Weyl's dimension formula (Theorem~\ref{thm:Weyl_dim}), which directly yields that the irreducible representation of highest weight $\lambda=(0,\ldots,0)$ has dimension 1. Now assume that $\lambda\neq(0,\ldots,0)$; we perform a change of variables
\begin{align*}
 m_1 & =\lambda_1-\lambda_2+1,\\
 m_2 & = \lambda_2-\lambda_3+1,\\
 \vdots & \\
 m_{N-1} & = \lambda_{N-1}-\lambda_N+1.
\end{align*}
The condition $(\lambda_1\geq\ldots\geq\lambda_N=0)$ is equivalent to $m_1,\ldots,m_{N-1}\geq 0$. We get
\[
d_\lambda=\prod_{1\leq i<j\leq N} \frac{m_i+\ldots+m_{j-1}}{j-i}.
\]
This expression is clearly nondecreasing in the coefficients $m_i$, therefore the smallest possible values of $d_\lambda$ are those of $\lambda_1=(1,0,\ldots,0)$, $\lambda_2=(1,1,0,\ldots,0)$, $\ldots$, $\lambda_{N-1}=(1,\ldots,1,0)$. In other terms,
\[
\inf\left\lbrace d_\lambda, \lambda\in\widehat{\SU}(N)\setminus\{(0,\ldots,0)\}\right\rbrace\geq \inf_{1\leq i\leq N-1}d_{\lambda_i} = \inf_{1\leq i\leq N-1}\binom{N}{i}=N.
\]
Thus, for any nontrivial highest weight $\lambda$, we find that $d_\lambda\geq N$. From this point the proof becomes elementary: for any $1<s'<s$,
\[
\zeta_{\SU(N)}(s)-1=\sum_{\substack{\lambda\in\widehat{\SU}(N)\\ \lambda\neq(0,\ldots,0)}}d_\lambda^{-s} = \sum_{\substack{\lambda\in\widehat{\SU}(N)\\ \lambda\neq(0,\ldots,0)}} d_\lambda^{-s'-(s-s')}\leq N^{-(s-s')}\sum_{\substack{\lambda\in\widehat{\SU}(N)\\ \lambda\neq(0,\ldots,0)}} d_\lambda^{-s'},
\]
and the RHS converges to 0 as $N$ tends to infinity.
\end{proof}

The previous lemma, which has many consequences, was proved for a large class of groups by Guralnick--Larsen--Manack \cite{GLM} and rediscovered by the author later. It will imply the convergence of the partition function $Z_{\U(N)}(g,t)$ for $g\geq 2$. Before giving the details, let $\theta:\C\to\C$ be the Jacobi theta function defined by the series
\[
\theta(z)=\sum_{n\in\Z} z^{n^2}.
\]
Note that it is in particular holomorphic on $\{z\in\C:\vert z\vert<1\}$. The asymptotics of partition functions on genus $g\geq 2$ are the following, as first stated in \cite{Rus} and proved in \cite{Lem}.

\begin{theorem}
For any $g\geq 2$ and $t>0$, we have the following convergence:
\begin{equation}
\lim_{N\to\infty}Z_{\U(N)}(g,t)=\theta(q_t),
\end{equation}
by setting $q_t=e^{-\frac{t}{2}}$.
\end{theorem}

\begin{proof}
We use the following bijection:
\[
\left\lbrace
\begin{array}{ccc}
\widehat{\SU}(N)\times\Z & \longrightarrow & \widehat{\U}(N)\\
(\lambda,n) & \longmapsto & \lambda+n:=(\lambda_1+n,\ldots,\lambda_N+n).
\end{array}
\right.
\]
It is easy to check that $d_{\lambda+n}=d_\lambda$ for all $\lambda\in\widehat{\SU}(N)$ and $n\in\Z$. Consequently,
\[
Z_{\U(N)}(g,t)=\sum_{\lambda\in\widehat{\SU}(N)}\sum_{n\in\Z}q_t^{c_2(\lambda+n)} d_{\lambda+n}^{2-2g}=\sum_{\lambda\in\widehat{\SU}(N)}\left(\sum_{n\in\Z}q_t^{c_2(\lambda+n)}\right)d_\lambda^{2-2g}.
\]
We separate the term $(0,\ldots,0)$ in the sum over $\lambda$, and notice that $c_2(n,\ldots,n)=n^2$:
\[
Z_{\U(N)}(g,t)=\sum_{n\in\Z}q_t^{n^2}+\sum_{\substack{\lambda\in\widehat{\SU}(N)\\ \lambda\neq(0,\ldots,0)}}\sum_{n\in\Z}q_t^{c_2(\lambda+n)}d_\lambda^{2-2g}.
\]
However,
\[
c_2(\lambda+n)=\frac1N\left(\sum_i (\lambda_i+n)^2 + \sum_{i<j}(\lambda_i-\lambda_j)\right)\geq n^2,
\]
given that $\lambda_1\geq\ldots\geq\lambda_N=0$. Finally,
\[
0\leq Z_{\U(N)}(g,t)-\theta(q_t) \leq \theta(q_t)\sum_{\substack{\lambda\in\widehat{\SU}(N)\\ \lambda\neq(0,\ldots,0)}}d_\lambda^{2-2g},
\]
and the RHS converges to zero by Lemma~\ref{lem:Zeta}, because $2g-2>1$ for any $g\geq 2$.
\end{proof}

\begin{remark}
The only highest weights of $\U(N)$ that contribute to the limit are the constant highest weights $(n,\ldots,n)$ with $n\in\Z$. If we remember that $\widehat{\U}(1)=\Z$ and we compute the dimension and Casimir of irreducible representations of $\U(1)$, we find that
\[
Z_{\U(1)}(g,t)=\theta(q_t),\quad \forall g\geq 0,\ \forall t>0.
\]
The previous theorem can thus be interpreted as follows: for all $g\geq 2$ and $t>0$,
\[
\lim_{N\to\infty}Z_{\U(N)}(g,t)=Z_{\U(1)}(g,t).
\]
The partition function becomes asymptotically the one of an abelian gauge theory, which is a phenomenon that one can also recover in another regime, where $N$ is fixed and we let $g\to\infty$ \cite{Lem3}. As we will discover in the next subsections, it is no longer the case if $g\leq 1$.
\end{remark}

\subsection{Genus 1: Discrete Gaussian measure in a cone and random partitions}

In the case $g=1$, we see that
\[
Z_{\U(N)}(1,t)=\sum_{\lambda\in\widehat{\U}(N)}e^{-\frac{t}{2}c_2(\lambda)},
\]
therefore all techniques from the previous subsection become useless because they were based on lower bounds of dimensions. We thus require an understanding of the asymptotics of Casimirs. A probabilistic approach, proposed in \cite{LM}, is the following: consider for any real $t>0$ and any integer $N\geq 1$ the probability measure $\Gfr_{N,t}$ on $\widehat{\U}(N)$ defined by
\begin{equation}
\Gfr_{N,t}(\lambda) = \frac{1}{Z_{\U(N)}(1,t)}e^{-\frac{t}{2}c_2(\lambda)},\quad \forall \lambda\in\widehat{\U}(N).
\end{equation}
If we recall that the Casimir is a quadratic expession of $\lambda$, the measure can be interpreted as a discrete Gaussian measure on $\widehat{\U}(N)$, which is a cone of $\Z^N$. The interpretation becomes even more obvious in the abelian case: on $\widehat{\U}(1)=\Z$ we have
\[
\Gfr_{1,t}(n)=\frac{1}{\theta(q_t)}e^{-\frac{t}{2}n^2},
\]
where $q_t=e^{-\frac{t}{2}}$. We will denote by $\E_{N,t}$ the expectation with respect to $\Gfr_{N,t}$. In the case $N=1$, the moments are fairly explicit.

\begin{proposition}
For any $t>0$ and $k\in\N$,
\begin{equation}
\E_{1,t}[n^{2k}] = \frac{1}{\theta(q_t)}\left(q\frac{d}{dq}\right)^k \theta(q_t),\quad 
\E_{1,t}[n^{2k+1}] = 0.
\end{equation}
\end{proposition}

\begin{proof}
For odd moments, it is a consequence of the fact that $n\mapsto n^{2k+1}e^{-\frac{t}{2}n^2}$ is an odd function on $\Z$. For even moments, we can obtain the formula by induction from
\[
\left(q\frac{d}{dq}\right)q^{n^2}=n^2q^{n^2}.
\]
\end{proof}

Surprisingly, it has been discovered that the measure $\Gfr_{N,t}$ can be expressed in terms of $\Gfr_{1,t}$ and a measure on the set of integer partitions. Recall that a partition of an integer $n$ is a nonincreasing family $\alpha=(\alpha_1,\ldots,\alpha_r)$ of positive integers: $\alpha_1\geq\ldots\geq\alpha_r>0$. Its \emph{length} is denoted by $\ell(\alpha)=r$ and its \emph{size} by $\vert\alpha\vert=n$. One can also write $\alpha\vdash n$. Partitions can be represented by Young diagrams, which are piles of cells, aligned in the left, such that the $i$-th row contains $\alpha_i$ cells (see Figure~\ref{fig:young_diagram}).
\begin{figure}[b!]
    \centering
    \includegraphics[width=0.4\linewidth]{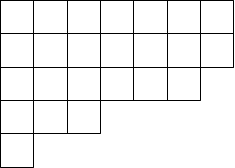}
    \caption{The Young diagram of the partition $(7,7,6,3,1)$.}
    \label{fig:young_diagram}
\end{figure}

The set $\Pfr_n$ of partitions of $n$ is in bijection with the dual of the symmetric group $\widehat{S}_n$, and the set of all partitions is the disjoint union $\Pfr=\bigsqcup_{n\geq 1}\Pfr_n$. For any $q\in(0,1)$, let the $q$-uniform measure on $\Pfr$ be the probability measure
\[
\Ufr_q(\alpha)=\phi(q)q^{\vert\alpha\vert},\quad \forall \alpha\in\Pfr,
\]
where $\phi:(0,1)\to\R$ is defined as the convergent infinite product
\[
\phi(q)=\prod_{m=1}^\infty (1-q^m).
\]

A first ``naive" way to relate highest weights to partitions consists in shifting the coefficients of a highest weight so that all coefficients are nonnegative: if $\lambda=(\lambda_1\geq\cdots\geq\lambda_N)\in\widehat{\U}(N)$, then we can write $\lambda=\Phi_N(\alpha,n)$, where $\alpha\in\Pfr$ and $n\in\Z$ are defined by $n=\lambda_N$, $\ell(\alpha)=\max\{i\leq N-1:\ \lambda_i>n\}$, and
\[
\alpha_i = \lambda_i-\lambda_N,\quad \forall i\leq \ell(\alpha).
\]
We get a first bijection
\[
\widehat{\U}(N)\simeq\{\alpha\in\Pfr:\ell(\alpha)\leq N-1\}\times\Z.
\]
It is reminiscent of the construction used in the study of Witten zeta functions, if we note that
\[
\widehat{\SU}(N)\simeq\{\alpha\in\Pfr:\ell(\alpha)\leq N-1\}.
\]
Let $\alpha=\alpha_\lambda$ be the partition obtained from $\lambda\in\widehat{\U}(N)$ from this construction. If we try to express $c_2(\lambda)$ in terms of $\alpha$ and $n$, we get
\begin{align*}
c_2(\lambda) = & \frac1N\left(\sum_i \lambda_i^2 + \sum_{1\leq i<j\leq N}\lambda_i-\lambda_j\right)\\
= & \frac1N\left(\sum_{i=1}^{\ell(\alpha)}(\alpha_i+n)^2+ (N-\ell(\alpha))n^2\right.\\
& + \left. \sum_{1\leq i<j\leq \ell(\alpha)}(\alpha_i-\alpha_j)+\left(N-\ell(\alpha)\right)\sum_{i=1}^{\ell(\alpha)}\alpha_i\right)\\
= & \vert\alpha\vert+n^2+\frac{1}{N}\left(\sum_{i=1}^{\ell(\alpha)}\alpha_i^2+(2n-\ell(\alpha))\vert\alpha\vert + \sum_{1\leq i<j\leq \ell(\alpha)}(\alpha_i-\alpha_j)\right).
\end{align*}
It is not clear how it behaves for large $N$, given a partition $\alpha$ of fixed size: for instance, consider two extremal cases, $\alpha_1=(1,\ldots,1,0)$ and $\alpha_2=(N-1,0,\ldots,0)$, so that $\ell(\alpha_1)=\vert\alpha_1\vert=N-1$, $\ell(\alpha_2)=1$ and $\vert\alpha_2\vert=N-1$. Then
\[
c_2(\Phi_N(\alpha_1,0))=2-\frac{2}{N},
\]
whereas
\[
c_2(\Phi_N(\alpha_2,0))=3N-6+\frac{3}{N},
\]
which shows that both partitions do not have comparable contributions. Intuitively, we expect that the highest weights that contribute to the limit are those which have small Casimirs, i.e. $c_2(\lambda)=O(1)$ as $N\to\infty$. Actually, such highest weights are obtained by perturbation of a constant highest weight $(n,\ldots,n)$ with two small partitions $\alpha$ and $\beta$ as in Figure~\ref{fig:AFHW}. They have been used in particular in \cite{Lem,Lem3} under the name ``almost flat highest weights", and they are somehow related to the rational representations of $\GL(N,\C)$, see \cite{Koi} for instance.
\begin{figure}[b!]
    \centering
    \includegraphics[width=0.8\linewidth]{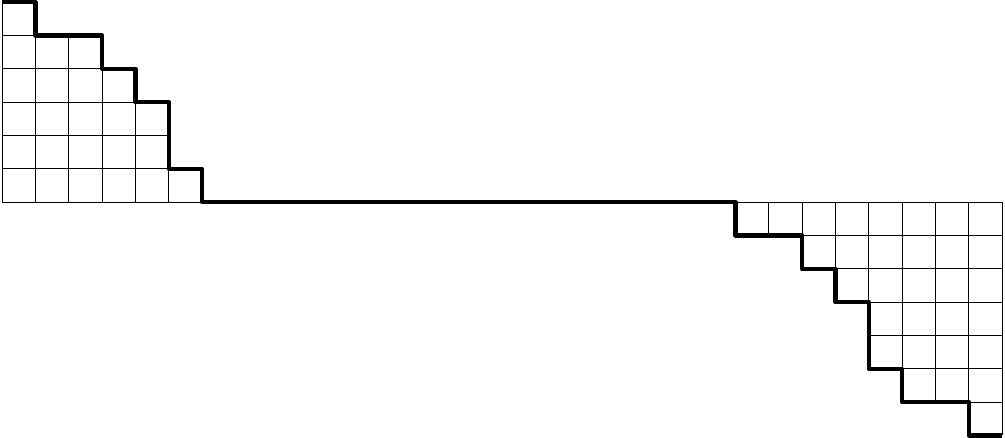}
    \caption{\small An almost flat highest weight.}
    \label{fig:AFHW}
\end{figure}
Let us start with a general statement. We shall use cutoff parameters $A_N=\lfloor(N+1)/2\rfloor-1$ and $B_N=N-\lfloor(N+1)/2\rfloor$, and remark that $A_N,B_N\sim N/2$ as $N\to\infty$.
\begin{proposition}
There exists a bijection
\[
\widehat{\U}(N)\simeq\Lambda_N:=\{(\alpha,\beta,n)\in\Pfr\times\Pfr\times\Z:\ell(\alpha)\leq A_N,\ell(\beta)\leq B_N\rfloor\}.
\]
It is given by
\[
\lambda_N(\alpha,\beta,n)=(\alpha_1+n,\ldots,\alpha_{\ell(\alpha)}+n,n,\ldots,n,n-\beta_{\ell(\beta)},\ldots,n-\beta_1).
\]
\end{proposition}

\begin{proof}
It is enough to exhibit the reciprocal bijection:
\[
(\lambda_N)^{-1}(\lambda)=(\alpha_\lambda,\beta_\lambda,n_\lambda),
\]
where $n_\lambda=\lambda_{\lfloor(N+1)/2\rfloor}$, and $\alpha_\lambda$ and $\beta_\lambda$ are the partitions defined by $(\alpha_\lambda)_i=\lambda_i-n_\lambda$, $(\beta_\lambda)_i=n-\lambda_{N-i}$.
\end{proof}
We will also need to introduce a new function on partitions, called the \emph{content}: if we denote by $\square$ a cell of the Young diagram corresponding to $\alpha$ and $R(\square)$ (resp. $C(\square)$) the index of its row (resp. column), then we set $c(\square)=C(\square)-R(\square)$. Let
\begin{equation}\label{eq:total_content}
K(\alpha)=\sum_{\square\in\alpha}c(\square)=\sum_{1\leq i\leq \ell(\alpha)}\sum_{1\leq j\leq \alpha_i}(j-i)
\end{equation}
be the \emph{total content} of the partition $\alpha$. The Casimir of the highest weight $\lambda_N(\alpha,\beta,n)$ has a very convenient expression:

\begin{corollary}[\cite{Lem}]\label{cor:Casimir}
For any $(\alpha,\beta,n)\in\Lambda_N$,
\begin{equation}
c_2(\lambda_N(\alpha,\beta,n))=\vert\alpha\vert+\vert\beta\vert+n^2+\frac{2}{N}(K(\alpha)+K(\beta)+n(\vert\alpha\vert-\vert\beta\vert)).
\end{equation}
\end{corollary}

The following result, which is one of the main results of \cite{LM}, shows that a random highest weight $\lambda\in\widehat{\U}(N)$ distributed according to $\Gfr_{N,t}$ can be seen as a coupling of two $q_t$-uniform random partitions and an integer distributed according to $\Gfr_{1,t}$, where $q_t=e^{-\frac{t}{2}}$.

\begin{theorem}[\cite{LM}]\label{thm:highest_weights_partitions}
Let $N\geq 1$ be an integer and $t>0$ be a real number. If $\lambda\sim\Gfr_{N,t}$, then for any test function $F:\widehat{\U}(N)\to\R$,
\begin{align}
\begin{split}
\E_{N,t}[F(\lambda)] = & \frac{\theta(q_t)}{Z_{\U(N)}(1,t)\phi(q_t)^2}\\
& \times \E[F(\lambda_N(\alpha,\beta,n))q_t^{\frac{2}{N}(K(\alpha)+K(\beta)+n(\vert\alpha\vert-\vert\beta\vert)}\mathbf{1}_{\Lambda_N}(\alpha,\beta,n)],
\end{split}
\end{align}
where $(\alpha,\beta,n)$ are independent random variables such that $\alpha,\beta$ are $q_t$-uniform and $n\sim\Gfr_{1,t}$.
\end{theorem}

The proof of this theorem is a direct consequence of the definitions and of Corollary~\ref{cor:Casimir}. It enables the following result:

\begin{corollary}
For any $t>0$,
\begin{equation}
\lim_{N\to\infty} Z_{\U(N)}(1,t)=\frac{\theta(q_t)}{\phi(q_t)^2}.
\end{equation}
\end{corollary}

This limit was first written in physics \cite{Dou}, and proved in \cite{Lem} by combinatorial techniques. It can be recovered from the previous theorem by taking $F(\lambda)=1$ for any $\lambda\in\widehat{\U}(N)$ and by using two key estimates: a global bound on $c_2(\lambda_N(\alpha,\beta,n))$ that enables to perform a (slightly modified) dominated convergence, and deviation inequalities for the measure $\Ufr_q$ for any $q\in(0,1)$ that ensure that $\E[\mathbf{1}_{\Lambda_N}(\alpha,\beta,n)]=1+O(e^{-\mathrm{cst}\cdot N})$. As we will see later, this framework has much more to offer, because it can lead to an asymptotic expansion up to any order, as well as an interpretation in terms of random surfaces.

\subsection{Genus 0: Douglas--Kazakov phase transition}

The careful reader will observe that we have seen the asymptotics of $Z_{\U(N)}(g,t)$ for any $g\geq 1$, but the case of the sphere remains to be treated. Surprisingly, the corresponding regime is completely different and requires another approach. From Migdal's formula, one can prove that
\begin{equation}
Z_{\U(N)}(0,t)=\frac{e^{\frac{t}{24}(N^2-1)+N(N-1)\log N}}{\prod_{1\leq i<j\leq N}(j-i)^2}\sum_{\lambda\in\widehat{\U}(N)}e^{-N^2\mathcal{J}_t(\widehat{\mu}_\lambda)},
\end{equation}
where $\widehat{\mu}_\lambda$ is the empirical measure of rescaled highest weights:
\[
\widehat{\mu}_\lambda = \frac{1}{N}\sum_{i=1}^N\delta_{\lambda_i+\frac{N+1}{2}-i},
\]
and $\mathcal{J}_t$ is a functional on the space of Borel probability measures on $\R$ given by
\[
\mathcal{J}_t(\mu)=-\iint_{(x,y)\in\R^2,x\neq y}\log\vert x-y\vert\mu(dx)\mu(dy)+\frac{t}{2}\int_\R x^2\mu(dx).
\]
In this case, the partition function diverges in the large-$N$ limit, but the \emph{free energy} $\lim_{N\to\infty}\frac{1}{N^2}\log Z_{\U(N)}(0,t)$ exists and is finite. It was conjectured by Douglas and Kazakov \cite{DK} and proved by different methods by Boutet de Monvel and Shcherbina \cite{BS} and L\'evy and Ma\"ida \cite{LevMai2}.

\begin{theorem}[Douglas--Kazakov phase transition]
For all $t\geq 0$, the limit
\begin{equation}
F(t)=\lim_{N\to\infty} \frac{1}{N^2}\log Z_{\U(N)}(0,t)
\end{equation}
exists. It defines a function that is $\mathscr{C}^2$ on $[0,\infty)$ and $\mathscr{C}^\infty$ on $[0,\pi^2)\cup(\pi^2,\infty)$, whose third derivative has a discontinuity at $\pi^2$.
\end{theorem}

A popular method to compute the free energy nowadays is by using large deviation techniques. For instance, the proof from \cite{LevMai2} relies on \cite[Theorem 1.2]{GM}, which provides a quite general large deviation principle for random Hermitian matrices.

\section{Gauge/string duality on a torus}\label{sec:gauge-string}

Let $\Tbb$ be a compact two-dimensional torus. It is a compact surface of genus 1 as a smooth real manifold, but it is also a Riemann surface, i.e. a complex manifold of dimension 1. As such, it can be defined as the quotient $\C/\Lambda$ where $\Lambda=\Z\oplus\tau\Z$ is a lattice generated by $(1,\tau)$ with $\Im(\tau)>0$. 

\begin{definition}
Let $R=\{x_1,\ldots,x_k\}$ be a finite subset of $\Tbb$ and $n\geq 1$ be an integer. A \emph{ramified covering} of $\Tbb$ of degree $n$ with ramification locus $R$ is a Riemann surface $X$ together with a continuous surjective map $\pi:X\to\Tbb$ such that:
\begin{enumerate}
\item For any $x\in\Tbb\setminus R$, the preimage of an open neighborhood $U\subset \Tbb\setminus R$ of $x$ by $\pi$ is the disjoint union of open sets $V_1,\ldots,V_n$ of $X$,
\item For all $x\in R$, there is an integer $n_x\geq 2$ such that $\pi^{-1}(\{x\})$ contains $n+1-n_x$ points (instead of $n$). The number $n_x$ is the \emph{ramification order}, and $x$ is called a \emph{generic ramification point} if $n_x=2$.
\end{enumerate}
\end{definition}

Fix a basepoint $x_0\in\Tbb$. For simplicity, we shall always assume that coverings are not ramified over $x_0$ (otherwise we can change the basepoint). The monodromy of the ramified covering $X\to\Tbb$ along loops based at $x_0$, which is basically the same object as holonomy for principal $G$-bundles, yields a bijection between the set of ramified coverings of degree $n$ with ramification locus $R=\{x_1,\ldots,x_k\}$ and the space
\[
\Hom(\pi_1(\Tbb\setminus R,x_0),S_n).
\]
The group of automorphisms of such coverings is $S_n$, and the space of equivalence classes of ramified coverings of degree $n$ with $k$ ramification points of $\Tbb$ is
\[
\mathcal{R}_{n,k}=\Hom(\pi_1(\Tbb\setminus R,x_0),S_n)/S_n,
\]
where $R=\{x_1,\ldots,x_k\}\subset \Tbb$ is an \emph{arbitrary} finite subset of size $k$ (indeed, the fundamental groups of $\Tbb\setminus\{x_1,\ldots,x_k\}$ and $\Tbb\setminus\{x'_1,\ldots,x'_k\}$ for any $\{x_1,\ldots,x_k\}$ and $\{x'_1,\ldots,x'_k\}$ are isomorphic). To avoid cumbersome notations, we will also denote by $X$ the equivalence class of the corresponding covering $X\in\Hom(\pi_1(\Tbb\setminus R,x_0),S_n)$. Remark that $\mathcal{R}_{n,k}$ is analogous to character varieties introduced in the end of \S\ref{sec:loop_group}, but with the finite group $G=S_n$. The enumeration of ramified coverings is usually considered up to equivalence, which corresponds to the counting measure on $\mathcal{R}_{n,k}$ (and more specifically, on subsets with specific ramification profiles, see \cite[\S 2.1]{EO}. We can form the countably infinite space
\[
\mathcal{R}=\bigsqcup_{n,k\geq 0}\mathcal{R}_{n,k}
\]
of all possible ramified coverings of the torus, up to automorphisms. The counting measure on $\mathcal{R}$ is $\mu$ such that
\[
\int_\mathcal{R}f(X)\mu(dX)=\sum_{n,k\geq 0}\sum_{X\in\mathcal{R}_{n,k}}f(X)=\sum_{n,k\geq 0}\sum_{\substack{X\to\Tbb\\ \deg(X)=n\\ \chi(X)=-k}}\frac{f(X)}{n!},
\]
where $\chi(X)$ is the Euler characteristic of of $X\to\Tbb$. Any random model built on $\mathcal{R}$ would be a model of random surfaces, suitable for a string theory (although finding the corresponding Lagrangian, or string action, might be a challenge), and we will describe one in this subsection: for any $n,k$ we consider the \emph{Hurwitz space} $\mathcal{H}_1(n,k)\subset\mathcal{R}_{n,k}$ of (orientable) ramified coverings of degree $n$ with $k$ \emph{generic} ramifications: the monodromy of a loop around any ramification point is in the conjugacy class of a transposition. It is one of the simplest models of random ramified coverings, and also a very natural one because it is related to random walks on the Cayley graph of the symmetric group generated by transpositions \cite{Nov24}.

By the Riemann--Hurwitz formula, one can check that $\mathcal{H}_1(n,2k+1)$ is empty for any $k$, and that any $X\in\mathcal{H}_1(n,2k)$ has genus $g=k+1$. Counting such ramified coverings is a standard problem in enumerative geometry, considered for instance by Eskin and Okounkov \cite{EO}, and we denote by $H_1(n,2k)=\#\mathcal{H}_1(n,2k)$ the size of the Hurwitz space, also known as a \emph{Hurwitz number}. Let
\[
\mathcal{F}_k(q)=\sum_{n\geq 1}q^n H_1(n,2k)
\]
be the generating function of such numbers. Using a general formula due to Frobenius \cite[Theorem A.1.10]{LZ}, we have:
\[
\sum_{\alpha\in\Pfr} q^{\vert\alpha\vert}K(\alpha)^{2k}=\mathcal{F}_k(q),
\]
where $K(\alpha)$ is the total content introduced in~\eqref{eq:total_content}. By similar arguments,
\[
\sum_{\alpha\in\Pfr} q^{\vert\alpha\vert}\vert\alpha\vert^\ell K(\alpha)^{2k}=\sum_{n\geq 1}q^nn^\ell H_1(n,2k)=\left(q\frac{d}{dq}\right)^k\mathcal{F}_k(q).
\]
Using Theorem~\ref{thm:highest_weights_partitions} and expanding the exponential terms up to order $p\geq 1$, we obtain (up to technical details that are developed in \cite{LM}) the following:
\begin{theorem}
For any $p\geq 1$, and any $t>0$, we have an asymptotic expansion
\begin{equation}\label{eq:asympt_exp}
Z_{\U(N)}(1,t)=a_0(t)+\frac{a_1(t)}{N^2}+\ldots+\frac{a_{p}(t)}{N^{2p}}+O\left(N^{-2p-2}\right).
\end{equation}
The coefficients $(a_k(t))_{k\geq 0}$ do not depend on $p$, and they are explicit functionals of $\mathcal{F}_k(q_t)$ and $\theta(q_t)$.
\end{theorem}
In a certain sense, the asymptotic expansion~\eqref{eq:asympt_exp} can be rewritten as an integral over $\mathcal{R}^2$ with respect to a measure that randomizes the number of generic ramification points and the degree of ramified coverings, and such that the ramified coverings, conditioned to have degree $n$ and number of (generic) ramification points $2k$, are uniform in the Hurwitz space $\mathcal{H}_1(n,2k)$. The precise measure (which has infinite mass, hence it cannot be normalized) is
\[
\rho_t(dX)=\sum_{n\geq 0}e^{-\frac{t}{2}n}\sum_{k\geq 1}\frac{t^{2k}}{(2k)!}\sum_{Y\in\mathcal{H}_1(n,2k)}\delta_{Y}(dX),
\]
where the Dirac mass is applied to equivalence classes in $\mathcal{R}$. We obtain two random ramified coverings $X_1,X_2\to\Tbb$ which are coupled by a random integer $n$, and this coupling reflects the coupling of random partitions obtained in Theorem~\ref{thm:highest_weights_partitions}.

\begin{theorem}[\cite{LM2}]
For any real $t>0$, and any integers $p\geq 0$ and $N\gg 1$ large enough, there is an explicit functional $\Phi_{p,t,N}:\mathcal{R}^2\to\R$ such that
\begin{equation}
Z_{\U(N)}(1,t)=\int \Phi_{p,t,N}(X,Y)N^{\chi(X)+\chi(Y)}\rho_t(dX)\rho_t(dY)+O(N^{-2p-2}),
\end{equation}
where the integral is over $(X,Y)\in\mathcal{R}^2$ such that $X$ and $Y$ only have generic ramifications, and $\chi(X)+\chi(Y)\geq -2p$.
\end{theorem}

Similar results hold for other groups, such as $\SU(N)$, $\mathrm{SO}(N)$ or $\mathrm{Sp}(N)$, and all proofs rely on a similar description of $Z_G(1,t)$ as the partition function of a Gaussian measure on $\widehat{G}$, or equivalently, the trace of the operator $e^{\frac{t}{2}\Delta_G}$ as a linear operator on $Z^2(G)$. There is also an equivalent formula where the cutoff is put on the degree of the coverings instead of their Euler characteristic.

To conclude, another promising direction towards the gauge/string duality is the use of \emph{Gromov--Witten invariants}. Indeed, they correspond to intersections of some cohomology classes of moduli spaces of maps $\Sigma\to M$, for a symplectic manifold $M$ and a Riemann surface $\Sigma$, and have been used in the last decades to describe correlation functions and partition functions of string theories from a more ``cohomological" point of view \cite{Kon95b,KM98}. When the target $M$ is a compact Riemann surface, Hurwitz theory coincides with a part of Gromov--Witten theory, according to a result by Okounkov--Pandharipande \cite{OP2}, and an application to Yang--Mills theory on a torus is also described in \cite{LM2}, but we will not elaborate further because it would lead us to a much more geometric realm.

\section*{Acknowledgements}

These notes combine the content of a talk given in 2024 at the International Congress of Mathematical Physics and a lecture on two-dimensional Yang--Mills theory given in 2025 at \'Ecole Normale Sup\'erieure-PSL, and the author is grateful to the respective organizations for the opportunities to speak there. The author also thanks Elias Nohra for a careful feedback and useful suggestions, Nguyen Viet Dang for enlightening discussions, as well as the anonymous referee for helpful comments that improved the clarity of the manuscript.

\bibliographystyle{amsplain}
\bibliography{YM2-survey}

\end{document}